\documentclass[12pt]{article}

\usepackage{amstext    }
\usepackage{amsthm    }
\usepackage{a4}
\usepackage[mathscr]{eucal}
\usepackage{mathrsfs}
\usepackage{indentfirst}
\usepackage{amsmath,bm}

\usepackage{amsmath}
\usepackage{amssymb}
\usepackage{amscd}

\let\Horig\H

\newtheorem{theorem}{Theorem}[section]

\newtheorem{proposition}[theorem]{Proposition}

\newtheorem{lemma}[theorem]{Lemma}
\newtheorem{remark}[theorem]{Remark}

\newcommand{\dist}{{\rm dist}}

\renewcommand{\H}{{\cal H}}

\title{Equidistribution on Big Line Bundles with Singular Metrics for Moderate Measures}

\author{Guokuan SHAO}

\begin{document}

\maketitle

\begin{abstract}
We establish an equidistribution theorem for the common zeros of random holomorphic sections of high powers of several singular Hermitian
big line bundles associated to moderate measures.
\end{abstract}

\noindent
{\bf Classification AMS 2010:} 32A60, 32L10, 32U40.


\section{Introduction}
Distribution of zeros of random polynomials is a classical subject.
The systematic research was developed in the papers \cite{bbl}, \cite{bp}, \cite{bs}, \cite{et}, \cite{ha}, \cite{ka}, \cite{sv}.
The general results about the holomorphic sections
of a positive line bundle associated to the Lebesgue measures induced by the Fubini-Study metric were obtained
by Shiffman-Zelditch\cite{sz}. Dinh-Sibony\cite{ds1} extended the equidistribution property
in the case of general measures and obtained a good estimate of the convergence speed.
The potential-theoretic approach from Forn\ae ss-Sibony\cite{fs} was used in the previous works.
Our paper \cite{sh} gave a large family of singular moderate measures that satisfies the equidistribution property.

The equidistribution property and convergence speed of holomorphic sections
of singular Hermitian holomorphic line bundles have been very active recently.
see \cite{cm1}, \cite{cm2}, \cite{cm3}, \cite{cmm}, \cite{cmn}, \cite{dmm} for complete references.
Dinh-Ma-Marinescu \cite{dmm} explored the equidistribution of zeros of random holomorphic sections for singular Hermitian metrics
with a convergence speed.
Coman-Marinescu-Nguy\^{e}n \cite{cmn} studied furtherly the equidistribution of common zeros of sections of several big line bundles.
The measure of the equidistribution theorem in \cite{cmn} is the standard one induced by the Fubini-Study metric.
On the other hand, our previous work \cite{sh} investigated the equidistribution of zeros of sections of a single positive holomorphic line bundle
associated to moderate measures. The metric on the line bundle in the latter work is smooth.
In this paper, the normalized currents are defined by the common zeros of \emph{m}-tuples of random sections of
high powers of \emph{m} singular Hermitian big line bundles on a compact K\"{a}hler manifold.
Using certain moderate measures we show that
these normalized currents distribute asymptotically to the wedge product of the curvature currents.
Consequently, we generalize both the main theorems in \cite{cmn} and \cite{sh}.

Our method follows the approach of Coman-Marinescu-Nguy\^{e}n \cite{cmn}.
Adapting their work, we prove that the intersections of currents of integration along subvarieties
are well-defined almost everywhere with respect to a finite product of moderate measures (see Section 2).
Moreover, their strategy of using Dinh-Sibony equidistribution theory leads us to obtain
an estimate for the convergence speed (see Section 3 and Section 4).
Then we combine the above with the technical analysis of moderate measures to achieve our results.
Here our hard core work consists of estimating efficiently some constants
which are intimately associated with multi-projective spaces (see Section 3).

We start with the basic settings of this paper. Let $X$ be a compact K\"{a}hler manifold of dimension $n$
with a fixed K\"{a}hler form $\omega$.
For every complex vector space $V$ of finite dimension,
let $\omega_{FS}$ be the standard K\"{a}hler form induced by the Fubini-Study metric on
its projective space $\mathbb{P}(V)$
normalized by $\int_{\mathbb{P}(V)}\omega_{FS}^{\dim\mathbb{P}(V)}=1$.
Recall that a singular Hermitian holomorphic line bundle $(L, h)$
is a holomorphic line bundle $L$ with a Hermitian metric which is given in any trivialization by a weight
function $e^{-\varphi}$ such that $\varphi$ is locally integrable. Let $c_{1}(L, h)$ be its curvature current
which represents the first Chern class. To be precise, if $e_{L}$ is a holomorphic frame of $L$ on an open subset
$U\subset X$, then $|e_{L}|_{h}^{2}=e^{-2\varphi}$, $c_{1}(L, h)=dd^{c}\varphi$ on $U$. Here $d=\partial+\bar\partial$,
$d^{c}=\frac{1}{2\pi i}(\partial-\bar\partial)$.
The case when $c_{1}(L, h)\geq 0$ as a current is particularly interesting.
We say that a holomorphic line bundle $L$ is $big$ if it admits a singular metric $h$ with that $c_{1}(L,h)\geq\epsilon\omega$
for some constant $\epsilon>0$, see \cite[Lemma 2.3.6]{mm1}.

Let $(L_{k}, h_{k})$, $1\leq k\leq m\leq n$, be $m$ singular Hermitian holomorphic line bundles on $X$.
Let $L_{k}^{p}$ be the $p$th tensor powers of $L_{k}$. Denote by $H_{(2)}^{0}(X, L_{k}^{p})$ the Bergman space
of $L^{2}$-holomorphic sections of $L_{k}^{p}$ relative to the metric $h_{k,p}:=h_{k}^{\otimes p}$ induced by $h_{k}$
and the volume form $\omega^{n}$ on $X$, endowed with the inner product
\begin{equation*}
\langle S, S' \rangle_{k,p}:=\int_{X}h_{k,p}(S, S')\omega^{n},
\end{equation*}
$\forall S, S'\in H_{(2)}^{0}(X, L_{k}^{p})$. Let $\mathbb{P}H_{(2)}^{0}(X, L_{k}^{p})$ be the associated projective space.
Set $d_{k,p}:=\dim H_{(2)}^{0}(X, L_{k}^{p})-1$.
We have $p^{n}/C\leq d_{k,p}\leq Cp^{n}$, where $C>0$ is a constant independent of $k$ and $p$ (see Theorem 4.1).
Now we consider the multi-projective space
\begin{equation}
\mathbb{X}_{p}:=\mathbb{P}H_{(2)}^{0}(X, L_{1}^{p})\times ...\times \mathbb{P}H_{(2)}^{0}(X, L_{m}^{p})
\end{equation}
endowed with a probability measure $\sigma_{p}$ for every $p\geq 1$.
Let $\pi_{k,p}: \mathbb{X}_{p}\rightarrow\mathbb{P}H_{(2)}^{0}(X, L_{k}^{p})$ be the natural projections.
Denote by $[S=0]$ the current defined by the zero set of $S\in H^{0}(X, L_{k}^{p})$.
Set
\begin{equation}
[S_{p}=0]:=[S_{p1}=0]\wedge...\wedge [S_{pm}=0], \quad\forall S_{p}=(S_{p1},...,S_{pm})\in\mathbb{X}_{p},
\end{equation}
whenever it is well-defined. Let
\begin{equation}
\mathbb{P}^{X}:=\prod_{p=1}^{\infty}\mathbb{X}_{p}.
\end{equation}
It is a probability space with the product measure $\sigma=\prod_{p=1}^{\infty}\sigma_{p}$.

A function $u: M\to\mathbb{R}$ defined on a metric space $(M, \rm dist)$
is said to be of class $\mathscr{C}^{\rho}$ for some exponent $0<\rho<1$, with modulus $c$ if
\begin{equation*}
\sup_{\substack{x,y\in M \\x\neq y}}\frac{|u(x)-u(y)|}{\dist(x,y)^{\rho}}\leq c.
\end{equation*}
Consider a complex manifold $M$ with a fixed volume form, let $\gamma$ be a closed real current of bidegree $(1,1)$ on $M$.
An upper-semi continuous function $u: M\to [-\infty, \infty)$ in $L^{1}_{loc}(M)$ is said to be $\gamma$-p.s.h. if $dd^{c}u+\gamma\geq 0$.
When $M$ is a compact K\"{a}hler manifold of dimension $N$, $\omega^{N}$ is the standard volume form,
we say that a function $\phi$ on $M$ is quasi-plurisubharmonic (q.p.s.h. for short) if it is $c\omega$-p.s.h. for some constant $c>0$.
Set
\begin{equation}
\mathcal{F}:=\{\phi ~q.p.s.h. ~on~ M: dd^{c}\phi\geq -\omega, \max_{M}\phi =0\}.
\end{equation}
Recall that a positive measure $\mu$ is $(c,\alpha)$-moderate for some constants $c>0, \alpha >0$ if
\begin{equation}
\int_{M}\exp(-\alpha \phi)d\mu\leq c
\end{equation}
for all $\phi\in\mathcal{F}$. The measure $\mu$ is called $moderate$ if it is $(c,\alpha)$-moderate for some constants $c>0, \alpha >0$.

We say that the analytic subsets $A_{1},...,A_{m}$ on $X$, $m\leq n$, are
in $general$ $position$ if codim$A_{i_{1}}\cap...\cap A_{i_{k}}\geq k$ for every $1\leq k\leq m$
and $1\leq i_{1}<...<i_{k}\leq m$.
We denote by $\dist$ the distance on $X$ induced by the fixed K\"{a}hler form $\omega$.
Let $\phi: U\rightarrow [-\infty, \infty)$ be a function on an open subset $U\subset X$,
$A\subset X$ a proper analytic subset.
Following the terminology in \cite{cmn}, $\phi$ is called {\it H\"{o}lder with singularities} along $A$
if there are positive constants $c, \delta$ and $0<\nu\leq 1$
such that
\begin{equation}
|\phi(z)-\phi(w)|\leq\frac{c\dist(z,w)^{\nu}}{\min\{\dist(z,A),\dist(w,A)\}^{\delta}}
\end{equation}
for all $z,w\in U\setminus A$.
A singular metric $h$ of $L$ is defined to be {\it H\"{o}lder with singularities} along $A$ if every local weight of $h$
is H\"{o}lder with singularities along $A$.
For motivations as well as examples of such metrics, we refer the readers to \cite{cmn}.

The multi-projective space $\mathbb{X}_{p}$ in \cite{cmn} is equipped with the
probability measure $\sigma_{p}^{0}$ which is the product of the Lebesgue measures induced by Fubini-Study metrics on the components.
In this paper, we define singular moderate measures $\sigma_{p}$ as perturbations of $\sigma_{p}^{0}$ on $\mathbb{X}_{p}$.
For each $p\geq 1, 1\leq k\leq m, 1\leq j\leq d_{k,p}$, let $u_{j}^{k,p}:\mathbb{P}H_{(2)}^{0}(X, L_{k}^{p})\rightarrow\mathbb{R}$
be an upper-semi continuous function. Fix $0<\rho<1$ and a sequence of positive constants $\{c_{p}\}_{p\geq 1}$.
We call $\{u_{j}^{k,p}\}$ {\it a family of $(c_{p},\rho)$-functions} if all $u_{j}^{k,p}$
satisfy the following two conditions:
\begin{flushleft}
$\bullet$ $u_{j}^{k,p}$ is of class $\mathscr{C}^{\rho}$ with modulus $c_{p}$, \\
$\bullet$ $u_{j}^{k,p}$ is a $c_{p}\omega_{FS}$-p.s.h.
\end{flushleft}
Then for each $p\geq 1$, there is a probability measure
\begin{equation}
\sigma_{p}=\prod_{k=1}^{m}\bigwedge_{j=1}^{d_{k,p}}\pi_{k,p}^{\star}(dd^{c}u_{j}^{k,p}+\omega_{FS})
\end{equation}
on $\mathbb{X}_{p}$.  By \cite[Theorem 1.1, Remark 2.12]{sh},
$\bigwedge_{j=1}^{d_{k,p}}(dd^{c}u_{j}^{k,p}+\omega_{FS})$ is a moderate measure on $\mathbb{P}H_{(2)}^{0}(X, L_{k}^{p})$
when $c_{p}\leq 1/c^{p^{n}}$ for a suitable constant $c>1$, $\forall 1\leq k\leq m, p\geq 1$.
We call
\begin{equation}
\sigma=\prod_{p=1}^{\infty}\sigma_{p}=\prod_{p=1}^{\infty}\prod_{k=1}^{m}\bigwedge_{j=1}^{d_{k,p}}\pi_{k,p}^{\star}(dd^{c}u_{j}^{k,p}+\omega_{FS})
\end{equation}
a probability measure on $\mathbb{P}^{X}$ generated by a family of $(c_{p},\rho)$-functions $\{u_{j}^{k,p}\}$ on $\{\mathbb{P}H_{(2)}^{0}(X, L_{k}^{p})\}$.

Here is our first main theorem.
\begin{theorem}
Let $(X, \omega)$ be a compact K\"{a}hler manifold of dimension $n$,
$(L_{k}, h_{k})$, $1\leq k\leq m\leq n$, be $m$ singular Hermitian holomorphic big line bundles on $X$.
The metric $h_{k}$ is continuous outside a proper analytic subset $A_{k}\subset X$,
$c_{1}(L_{k},h_{k})\geq\epsilon\omega$ on $X$ for some constant $\epsilon>0$,
and $A_{1},...,A_{m}$ are in general position. Let $0<\rho<1$. Then there exists a constant $c>1$ which depends only on $X, L_{k}, \rho$
with the following property: If $\sigma$ is a probability measure on $\mathbb{P}^{X}$ generated by a family of $(1/c^{p^{n}},\rho)$-functions $\{u_{k,p}^{j}\}$ on $\{\mathbb{P}H_{(2)}^{0}(X, L_{k}^{p})\}$ defined by (8), then for almost every $\{S_{p}\}_{p\geq 1}\in\mathbb{P}^{X}$ with respect to $\sigma$,
we have in the weak sense of currents as $p\rightarrow\infty$ on $X$,
\begin{equation*}
\frac{1}{p^{m}}[S_{p}=0]\rightarrow c_{1}(L_{1},h_{1})\wedge...\wedge c_{1}(L_{m},h_{m}).
\end{equation*}
\end{theorem}
\begin{remark}
When all $u_{k,p}^{j}\equiv 0$, then $\sigma_{p}$ are the Lebesgue measures $\sigma_{p}^{0}$
on $\mathbb{X}_{p}$ and we obtain \cite[Theorem 1.2]{cmn}.
In addition, the constant $c$ is independent of the choices of singular metrics on the big line bundles.
\end{remark}
When the metrics $h_{k}$ are all H\"{o}lder with singularities,
we can also extend the result in \cite{cmn} about the estimate of the speed of the above convergence
associated to the moderate measures defined by (7). Now we are able to state the second main theorem.
\begin{theorem}
We keep the notations and the hypotheses of Theorem 1.1. Suppose, moreover,
that $h_{k}$ is H\"{o}lder with singularities along $A_{k}$, $1\leq k\leq m$.
Then there exist a positive constant $\xi$ which depends only on $m$, and another positive constant $C$ which depends on
$X, (L_{1},h_{1}),...,(L_{m},h_{m})$ with the following property:
Given any sequence of positive numbers $\{\lambda_{p}\}_{p=1}^{\infty}$ with the following conditions
\begin{equation*}
\liminf_{p\rightarrow\infty}\frac{\lambda_{p}}{\log p}>(1+\xi n)C \quad \text{and} \quad
\lim_{p\rightarrow\infty}\frac{\lambda_{p}}{p^{n}}=0,
\end{equation*}
there exist subsets $E_{p}\subset\mathbb{X}_{p}$ such that for all $p$ sufficiently large,
\par (i)
\begin{equation*}
\sigma_{p}(E_{p})\leq Cp^{\xi n}\exp(-\frac{\lambda_{p}}{C}),
\end{equation*}
\par (ii) for any point $S_{p}\in\mathbb{X}_{p}\setminus E_{p}$ and any $(n-m,n-m)$-form $\phi$ of class $\mathscr{C}^{2}$,
\begin{equation*}
\bigl|\bigl<\frac{1}{p^{m}}[S_{p}=0]-\bigwedge_{k=1}^{m}c_{1}(L_{k},h_{k}), \phi \bigr> \bigr|\leq \frac{C\lambda_{p}}{p}\|\phi\|_{\mathscr{C}^{2}}.
\end{equation*}
\end{theorem}
\begin{remark}
The primary difference between Theorem 1.3 and \cite[Theorem 1.4]{cmn} is that the measures in Theorem 1.3 are moderate.
In addition, the convergence speed in Theorem 1.3 is of order $O(\log p/p)$ when $\lambda_{p}=C_{1}\log p$ for some constant $C_{1}>0$.
So the result generalizes \cite[Theorem 1.2]{dmm} and \cite[Theorem 1.2]{sh}.
\end{remark}

Note that in Theorem 1.1 and all other equidistribution theorems of the papers listed in this section, the limit point of
the convergence sequences can not be any positive closed current.
For example, there is a condition that $c_{1}(L_{k},h_{k})\geq \epsilon\omega$ in Theorem 1.1.
However, it is possible when $X=\mathbb{P}^{n}, L=\mathcal{O}(1), m=1$ with a probability measure $\sigma$ chosen properly.
\begin{theorem}
Given any positive closed current $T$ of bidegree $(1,1)$ with mass $1$ on $\mathbb{P}^{n}$.
Let $d_{p}=\dim H^{0}(\mathbb{P}^{n},\mathcal{O}(p))-1$.
Then there exists a family of smooth probability measures $\sigma_{p}=(\omega_{FS}+dd^{c}u_{p})^{d_{p}}$
on $\mathbb{P}H^{0}(\mathbb{P}^{n},\mathcal{O}(p))$ for some smooth real functions $u_{p}$ with the following property:
For almost every $S=(S_{p})\in\prod_{p\geq 1}\mathbb{P}H^{0}(\mathbb{P}^{n},\mathcal{O}(p))$
with respect to $\sigma=\prod_{p\geq 1}\sigma_{p}$, we have in the weak sense of currents as $p\rightarrow\infty$ on $\mathbb{P}^{n}$,
\begin{equation*}
\frac{1}{p}[S_{p}=0]\rightarrow T.
\end{equation*}
\end{theorem}

The paper is organized as follows. In Section 2 we recall the notion of Fubini-Study currents and prove a Bertini type theorem
associated to moderate measures. In Section 3 we first introduce several results of meromorphic transforms induced by Kodaira maps.
This enables us to estimate some constants
intimately associated to moderate measures.
Consequently, we establish relevant estimates on multi-projective spaces.
Section 4 is devoted to prove the main theorems.

\section{Intersection of Fubini-Study currents and Bertini theorem for moderate measures}
In this section, we introduce some results about
the intersection of the Fubini-Study currents associated to $m$ line bundles.
We will see that the current $c_{1}(L_{1},h_{1})\wedge...\wedge c_{1}(L_{m},h_{m})$ in Theorem 1.1 is well-defined.
Based on the elementary techniques in \cite{cmn}, we also show that for almost all
the zero-divisors of sections of high powers of the bundles with respect to
a moderate measure are in general position. Then it follows from \cite{de} that
the currents $[S_{p}=0]$ in Theorem 1.1 are well-defined for almost all $S_{p}$ with respect to moderate measures $\sigma$ in Theorem 1.1.

We keep the notations and hypotheses in Section 1.
Consider the $Kodaira$ $map$
\begin{equation*}
\Phi_{k,p}: X\rightarrow\mathbb{P}(H_{(2)}^{0}(X, L_{k}^{p})^{\star}).
\end{equation*}
Here $H_{(2)}^{0}(X, L_{k}^{p})^{\star}$ is the dual space of $H_{(2)}^{0}(X, L_{k}^{p})$.
Choose $\{S_{k,p}^{j}\}_{j=0}^{d_{k,p}}$ as an orthonormal basis of $H_{(2)}^{0}(X, L_{k}^{p})$.
By an identification via the basis, it boils down to a meromorphic map
\begin{equation*}
\Phi_{k,p}: X\rightarrow \mathbb{P}^{d_{k,p}}.
\end{equation*}
Now we give a local analytic description of the above map.
Let $U\subset X$ be a contractible Stein open subset, $e_{k}$ be a local holomorphic frame of $L_{k}$ on $U$.
Then there exists a holomorphic function $s_{j}^{k,p}$ on $U$ such that $S_{k,p}^{j}=s_{j}^{k,p}e_{k}^{\otimes p}$.
Then the map is expressed locally as
\begin{equation}
\Phi_{k,p}(x)=[s_{0}^{k,p}(x):...:s_{d_{k,p}}^{k,p}(x)], \quad \forall x\in U
\end{equation}
It is called the Kodaira map defined by the basis $\{S_{k,p}^{j}\}_{j=0}^{d_{k,p}}$.
Denote by $P_{k,p}$ the Bergman kernel function defined by
\begin{equation}
P_{k,p}(x)=\sum_{j=0}^{d_{k,p}}|S_{k,p}^{j}(x)|^{2}_{h_{k,p}}, \quad |S_{k,p}^{j}(x)|^{2}_{h_{k,p}}=h_{k,p}(S_{k,p}^{j}(x),S_{k,p}^{j}(x)).
\end{equation}
It is easy to see that this definition is independent of the choice of basis.

Recall that $\omega_{FS}$ is the normalized Fubini-Study form on $\mathbb{P}^{d_{k,p}}$. We define the Fubini-Study
currents $\gamma_{k,p}$ of $H_{(2)}^{0}(X, L_{k}^{p})$ as pullbacks of $\omega_{FS}$ by Kodaira map,
\begin{equation}
\gamma_{k,p}=\Phi_{k,p}^{\star}(\omega_{FS}).
\end{equation}
We have in the local Stein open subset $U$,
\begin{equation*}
\gamma_{k,p}\bigl|_{U}=\frac{1}{2}dd^{c}\log\sum_{j=0}^{d_{k,p}}|s_{j}^{k,p}|^{2}.
\end{equation*}
This yields
\begin{equation*}
\frac{1}{p}\gamma_{k,p}=c_{1}(L_{k},h_{k})+\frac{1}{2p}dd^{c}\log P_{k,p}.
\end{equation*}
Since $\log P_{k,p}$ is a global function which belongs to $L^{1}(X, \omega^{n})$,
$\frac{1}{p}\gamma_{k,p}$ has the same cohomology class as $c_{1}(L_{k},h_{k})$.
The base locus of $H_{(2)}^{0}(X, L_{k}^{p})$ is denoted by
\begin{equation}
A_{k,p}=\{x\in X: S_{k,p}^{j}=0,  \forall 0\leq j\leq d_{k,p}\}, \quad 1\leq k\leq m.
\end{equation}
We have the following result \cite[Proposition 3.1]{cmn}.
\begin{proposition}
We keep the notations and hypotheses of Theorem 1.1, then
\par (i) For every $J\subset\{1,...,m\}, J'=\{1,...,m\}\setminus J$, the analytic subsets $A_{k,p}$ and $A_{l}$, for $k\in J, l\in J'$,
are in general position, when $p$ is sufficiently large.
\par (ii) The currents
\begin{equation*}
\bigwedge_{k\in J}\gamma_{k,p}\wedge\bigwedge_{l\in J'}c_{1}(L_{l},h_{l})
\end{equation*}
are well defined for every $J\subset\{1,...,m\}$, when $p$ is sufficiently large.
\end{proposition}

\begin{lemma}
Let $\sigma$ be a moderate measure on $\mathbb{P}^{N}$. Then every proper analytic subset of $\mathbb{P}^{N}$
has measure zero with respect to $\sigma$.
\end{lemma}

\begin{proof}
By the homogeneity of $\mathbb{P}^{N}$, it is sufficient to prove that
\begin{equation*}
\sigma([z_{0}])=0
\end{equation*}
for some homogeneous coordinate $[z_{0},...,z_{N}]$.
By the definition of a moderate measure (cf. (5)), there exist constants $c>0, \alpha>0$ such that
\begin{equation*}
\int_{\mathbb{P}^{N}}\exp(-\alpha\phi)d\sigma\leq c,
\end{equation*}
$\forall \phi\in\mathcal{F}$, where $\mathcal{F}$ is defined in (1).
It follows that
\begin{equation*}
\int_{\mathbb{P}^{N}}|\phi|d\sigma<\infty.
\end{equation*}
Let $\phi=\log\frac{|z_{0}|}{|z|}$, where $|z|^{2}=\sum_{j=0}^{N}|z_{j}|^{2}$.
This function is well defined on $\mathbb{P}^{N}$. Note that $\phi$ is $\omega_{FS}$-p.s.h., $\max_{z\in\mathbb{P}^{N}}\phi(z)=0$.
Then $\phi\in\mathcal{F}$. So we have that
\begin{equation*}
\int_{[z_{0}=0]}\bigl |\log\frac{|z_{0}|}{|z|}\bigr |d\sigma<\infty.
\end{equation*}
Hence $\sigma([z_{0}])=0$.
The proof is completed.
\end{proof}

Adapting the proof of \cite[Proposition 3.2]{cmn}, we obtain the following Bertini type theorem in the context of moderate measures.
\begin{proposition}
Let $X$ be a compact complex manifold of dimension $n$. Let $L_{k}, 1\leq k\leq m\leq n$, be $m$ holomorphic line bundles
on $X$. If
\par (i) $V_{k}$ is a vector subspace of $H^{0}(X, L_{k})$ with basis $S_{k,0},...,S_{k,d_{k}}$, the base loci
$Bs V_{1},...,Bs V_{m}$ are in general position, where $Bs V_{k}:=\{x\in X: S_{k,0}(x)=...=S_{k,d_{k}}(x)=0\}$.
\par (ii) For each $t_{k}=[t_{k,0}:...:t_{k,d_{k}}]\in\mathbb{P}^{d_{k}}$, we set
\begin{equation}
Z(t_{k}):=\{x\in X: \sum_{j=0}^{d_{k}}t_{k,j}S_{k,j}(x)=0\}.
\end{equation}
\par (iii) $\sigma=\sigma_{1}\times...\times\sigma_{m}$ is the product measure on the multi-projective space
$\mathbb{P}^{d_{1}}\times...\times\mathbb{P}^{d_{m}}$, where $\sigma_{k}$ is a probability moderate measure on $\mathbb{P}^{d_{k}}$.
\par Then the analytic subsets $Z(t_{1}),...,Z(t_{m})$ are in general position for almost all $(t_{1},...,t_{m})\in\mathbb{P}^{d_{1}}\times...\times\mathbb{P}^{d_{m}}$ with respect to $\sigma$.
\end{proposition}

\begin{proof}
Let $\sigma_{l_{1}...l_{k}}=\sigma_{l_{1}}\times...\times\sigma_{l_{k}}$ be the product measure
on $\mathbb{P}^{d_{l_{1}}}\times...\times\mathbb{P}^{d_{l_{k}}}$ for every $1\leq l_{1}<...<l_{k}\leq m$.
Set
\begin{equation*}
U_{k}=\{(t_{l_{1}},...,t_{l_{k}})\in\mathbb{P}^{d_{l_{1}}}\times...\times\mathbb{P}^{d_{l_{k}}}: \dim Z(t_{l_{1}})\cap...\cap Z(t_{l_{k}})\cap \tilde V_{j}\leq n-k-j\},
\end{equation*}
where $1\leq l_{1}<...<l_{k}\leq m, \tilde V_{0}=X$, $\tilde V_{j}=Bs V_{i_{1}}\cap...\cap Bs V_{i_{j}}$ for some $i_{1}<...<i_{j}$
in $\{1,...,m\}\setminus\{l_{1},...,l_{k}\}$ if $1\leq j\leq m-k$. Note that the sets $U_{k}$ depend on the choices of $l_{1},...,l_{k},j$
and $\tilde V_{j}$.

It is sufficient to prove that
\begin{equation*}
\sigma_{l_{1}...l_{k}}(U_{k})=1
\end{equation*}
by induction on $k$ for every subset $U_{k}$, $1\leq l_{1}<...<l_{k}\leq m, 0\leq j\leq m-k$.
We only consider the case when $\{l_{1},...,l_{k}\}=\{1,...,k\}$. Write $\sigma_{k}'=\sigma_{1...k}$ for short.
We first consider the case when $k=1$. If $j=0$, then
\begin{equation*}
U_{1}=\{t_{1}\in\mathbb{P}^{d_{1}}: \dim Z(t_{1})\leq n-1\}=\mathbb{P}^{d_{1}}.
\end{equation*}
If $1\leq j\leq m-1$, let $\tilde V_{j}=\bigcup_{l=1}^{N}D_{l}\cup B$,
where the subsets $D_{l}$ are the irreducible components of $\tilde V_{j}$ of dimension $n-j$ and $B$ is of dimension less than $n-j$.
So $\{t_{1}\in\mathbb{P}^{d_{1}}: D_{l}\subset Z(t_{1})\}$ is a proper linear subspace of $\mathbb{P}^{d_{1}}$.
If not, $D_{l}\subset Bs V_{1}$ implies that $\dim \tilde V_{j}\cap Bs V_{1}=n-j$,
which contradicts the condition that $Bs V_{1},...,Bs V_{m}$ are in general position.
We know that $\dim Z(t_{1})\cap \tilde V_{j}\geq n-j$ if $t_{1}\in\mathbb{P}^{d_{1}}\setminus U_{1}$.
Since $Z(t_{1})\cap \tilde V_{j}$ is an analytic subset in $\tilde V_{j}$, $D_{l}\subset Z(t_{1})\cap \tilde V_{j}$ for some $l$.
It follows that
\begin{equation*}
\mathbb{P}^{d_{1}}\setminus U_{1}=\bigcup_{l=1}^{N}\{t_{1}\in\mathbb{P}^{d_{1}}: D_{l}\subset Z(t_{1})\}.
\end{equation*}
Hence we have that $\sigma_{1}(\mathbb{P}^{d_{1}}\setminus U_{1})=0$ by Lemma 2.2.

Now we can assume that $\sigma_{k}'(U_{k})=1$ for any $U_{k}$ defined as above. Set
\begin{equation*}
\begin{split}
U_{k+1}=\{&(t_{1},...,t_{k+1})\in\mathbb{P}^{d_{1}}\times...\times\mathbb{P}^{d_{k+1}}: \\
&\dim Z(t_{1})\cap...\cap Z(t_{k+1})\cap \tilde V_{j}\leq n-k-1-j \}, \\
U'=\{&(t_{1},...,t_{k})\in\mathbb{P}^{d_{1}}\times...\times\mathbb{P}^{d_{k}}: \\
&\dim Z(t_{1})\cap...\cap Z(t_{k})\cap \tilde V_{j}\leq n-k-j \}, \\
U''=\{&(t_{1},...,t_{k})\in\mathbb{P}^{d_{1}}\times...\times\mathbb{P}^{d_{k}}: \\
&\dim Z(t_{1})\cap...\cap Z(t_{k})\cap \tilde V_{j}\cap Bs V_{k+1}\leq n-k-1-j \}, \\
\end{split}
\end{equation*}
where $0\leq j\leq m-k-1, \tilde V_{0}=X, \tilde V_{j}=Bs V_{i_{1}}\cap...\cap Bs V_{i_{j}}$ for $k+2\leq i_{1}<...<i_{j}\leq m$
if $1\leq j\leq m-k-1$.
Let $U=U'\cap U''$. By induction on $k$, we know that $\sigma_{k}'(U')=\sigma_{k}'(U'')=1$, thus $\sigma_{k}'(U)=1$.

We need to prove that
\begin{equation*}
\sigma_{k+1}'(U_{k+1})=1.
\end{equation*}
It is enough to show that
\begin{equation*}
\sigma_{k+1}'(W)=0, \quad W:=(U\times\mathbb{P}^{d_{k+1}})\setminus U_{k+1}.
\end{equation*}
Given some $t=(t_{1},...,t_{k})\in U$, set
\begin{equation*}
\begin{split}
Z(t):&=Z(t_{1})\cap...\cap Z(t_{k}), \\
W(t):&=\{t_{k+1}\in\mathbb{P}^{d_{k+1}}: \dim Z(t)\cap \tilde V_{j}\cap Z(t_{k+1})\geq n-k-j\}, \\
\end{split}
\end{equation*}
then it is sufficient to show that $\sigma_{k+1}(W(t))=0$.

Let $Z(t)\cap \tilde V_{j}=\bigcup_{l=1}^{N}D_{l}\cup B$, where $D_{l}$ are irreducible components of $Z(t)\cap \tilde V_{j}$
of dimension $n-k-j$, $\dim B\leq n-k-1-j$ as $t\in U\subset U'$. By the same argument in the above,
if $t_{k+1}\in W(t)$, then $Z(t)\cap \tilde V_{j}\cap Z(t_{k+1})$ is an analytic subset of $Z(t)\cap \tilde V_{j}$ of dimension $n-k-j$,
hence there exists some $l$ such that $D_{l}\subset Z(t)\cap \tilde V_{j}\cap Z(t_{k+1})$.
We obtain that
\begin{equation*}
W(t)=\bigcup_{l=1}^{N} F_{l}(t), \quad F_{l}(t):=\{t_{k+1}\in\mathbb{P}^{d_{k+1}}: D_{l}\subset Z(t_{k+1})\}.
\end{equation*}
We claim that not all the sections of $V_{k+1}$ can vanish on $D_{l}$. If not, that is to say, $D_{l}\subset Bs V_{k+1}$,
this implies that
\begin{equation*}
\dim Z_{t}\cap \tilde V_{j}\cap Bs V_{k+1}=n-k-j,
\end{equation*}
which contradicts the fact that $t\in U''$.
Hence we can suppose that $S_{k+1,d_{k+1}}\not\equiv 0$ on $D_{l}$.
So
\begin{equation*}
\begin{split}
& F_{l}(t)\subset\{t_{k+1,0}=0\}\cup G_{l}(t), \\
& G_{l}(t):=\{[1:t_{k+1,1}:...:t_{k+1,d_{k+1}}]\in\mathbb{P}^{d_{k+1}}: D_{l}\subset Z([1:t_{k+1,1}:...:t_{k+1,d_{k+1}}])\}. \\
\end{split}
\end{equation*}
There exists at most one $\eta\in\mathbb{C}$ such that $[1:t_{k+1,1}:...:t_{k+1,d_{k+1}-1}:\eta]\in G_{l}(t)$
for any $(t_{k+1,1},...,t_{k+1,d_{k+1}-1})\in\mathbb{C}^{d_{k+1}-1}$. Otherwise, if there exist two complex numbers $\eta\neq\eta'$,
which satisfy the property, then we have on $D_{l}$,
\begin{equation*}
\begin{split}
S_{k+1,0}+t_{k+1,1}S_{k+1,1}+...+t_{k+1,d_{k+1}-1}S_{k+1,d_{k+1}-1}+\eta S_{k+1,d_{k+1}}&\equiv 0, \\
S_{k+1,0}+t_{k+1,1}S_{k+1,1}+...+t_{k+1,d_{k+1}-1}S_{k+1,d_{k+1}-1}+\eta' S_{k+1,d_{k+1}}&\equiv 0. \\
\end{split}
\end{equation*}
Then we have a contradiction with that $S_{k+1,d_{k+1}}\not\equiv 0$ on $D_{l}$.
This implies that $\sigma_{k+1}(G_{l}(t))=0$. Moreover, $\sigma_{k+1}(F_{l}(t))=0$.
Hence $\sigma_{k+1}(W(t))=0$. This completes the proof.
\end{proof}

In the setting of Theorem 1.1, let $V_{k,p}=H_{(2)}^{0}(X, L_{k}^{p})$ with orthonormal basis $\{S_{k,p}^{j}\}_{j=0}^{d_{k,p}}$.
Then the base locus of $V_{k,p}$ is $A_{k,p}$. Let $Z(t_{k})$ be an analytic hypersurface for any
$t_{k}=[t_{k,0}:...:t_{k,d_{k,p}}]\in\mathbb{P}^{d_{k,p}}$, defined in (13).
Let $\sigma_{p}$ be the product measure of probability moderate measures on
$\mathbb{P}^{d_{1,p}}\times...\times\mathbb{P}^{d_{m,p}}$ in Theorem 1.1 (cf. (7)).
Arguing as in the proof of \cite[Proposition 3.3]{cmn}, we obtain the following

\begin{proposition}
In the above setting,
\par (i) The analytic subsets $Z(t_{1}),...,Z(t_{m})$ are in general position, for almost every $(t_{1},...,t_{m})\in\mathbb{P}^{d_{1,p}}\times...\times\mathbb{P}^{d_{m,p}}$ with respect to $\sigma_{p}$, when $p$ is sufficiently large.
\par (ii) $Z(t_{i_{1}})\cap...\cap Z(t_{i_{k}})$ is of pure dimension $n-k$ for each $1\leq k\leq m, 1\leq i_{1}<...<i_{k}\leq m$.
\end{proposition}

\section{Dinh-Sibony equidistribution theorem and estimates on multi-projective spaces}
In this section, we introduce the techniques by Dinh-Sibony \cite{ds1}
and some results in \cite{cmn}. After giving some notions,
we prove the crucial estimates on multi-projective spaces associated to certain moderate measures.

Let $(X, \omega)$ (resp. $(Y, \omega_{Y})$) be a compact K\"{a}hler manifold of dimension $n$ (resp. $n_{Y}$).
Recall that a $meromorphic$ $transform$ $F: X\rightarrow Y$ is the graph of an analytic subset $\Gamma\subset X\times Y$
of pure dimension $n_{Y}+k$ such that the natural projections $\pi_{1}: X\times Y\rightarrow X$
and $\pi_{2}: X\times Y\rightarrow Y$ restricted to each irreducible component of the analytic subset $\Gamma$ are surjective.
We write $F=\pi_{2}\circ (\pi_{1}|_{\Gamma})^{-1}$.
The dimension of the fiber $F^{-1}(y):=\pi_{1}(\pi_{2}^{-1}|_{\Gamma}(y))$ is equal to $k$
for the point $y\in Y$ generic. This is the codimension of the meromorphic transform $F$.
If $T$ is a current of bidegree $(l,l)$ on $Y$, $n_{Y}+k-n\leq l\leq n_{Y}$, we define
\begin{equation}
F^{\star}(T):=(\pi_{1})_{\star}(\pi_{2}^{\star}(T)\wedge [\Gamma]),
\end{equation}
where $[\Gamma]$ is the current of integration over $\Gamma$.
We introduce the notations of intermediate degrees of $F$,
\begin{equation*}
\begin{split}
& d(F):=\int_{X}F^{\star}(\omega_{Y}^{n_{Y}})\wedge\omega^{k}, \\
& \delta(F):=\int_{X}F^{\star}(\omega_{Y}^{n_{Y}-1})\wedge\omega^{k+1}. \\
\end{split}
\end{equation*}

To introduce more notions and notations, we first recall the following lemma in \cite[Proposition 2.2]{ds1}.
\begin{lemma}
There exists a constant $r>0$ such that for any positive closed current $T$ of bidegree $(1,1)$ with mass $1$ on $(X, \omega)$,
there is a smooth $(1,1)$-form $\alpha$ which depends only on the cohomology class of $T$ and a q.p.s.h. function $\varphi$
satisfying that
\begin{equation*}
-r\omega\leq\alpha\leq r\omega, \quad dd^{c}\varphi-T=\alpha.
\end{equation*}
\end{lemma}
Denote by $r(X, \omega)$ the smallest $r$ in Lemma 3.1.
For example, $r(\mathbb{P}^{N},\omega_{FS})=1$.
Consider a positive measure $\mu$ on $X$. $\mu$ is said to be a PLB measure if all q.p.s.h. functions are integrable with respect to $\mu$.
It is easy to see that all moderate measures are PLB.
Now given a PLB probability measure $\mu$ on $X$ and $t\in\mathbb{R}$, we define,
\begin{equation}
\begin{split}
Q(X, \omega):&=\{\varphi ~q.p.s.h. ~on~ X, dd^{c}\varphi\geq -r(X,\omega)\omega\}, \\
R(X, \omega, \mu):&=\sup\{\max_{X}\varphi: \varphi\in Q(X, \omega), \int_{X}\varphi d\mu=0\} \\
&=\sup\{-\int_{X}\varphi d\mu: \varphi\in Q(X, \omega), \max_{X}\varphi=0\}, \\
S(X, \omega, \mu):&=\sup\{\bigl|\int\varphi d\mu \bigr|: \varphi\in Q(X, \omega), \int_{X}\varphi\omega^{n}=0\}, \\
\Delta(X, \omega, \mu, t):&=\sup\{\mu(\varphi<-t): \varphi\in Q(X, \omega), \int_{X}\varphi d\mu=0\}. \\
\end{split}
\end{equation}
These constants are related to Alexander-Dinh-Sibony capacity, see \cite[A. 2]{ds1} and \cite[Section 5]{gz}.

Let $\Phi_{p}$ be a sequence of meromorphic transforms from a projective manifold $(X, \omega)$ into
the compact K\"{a}hler manifolds $(\mathbb{X}_{p}, \omega_{p})$ of the same codimension $k$, where
$\mathbb{X}_{p}$ is defined in (1).
Let
\begin{equation*}
d_{0,p}=d_{1,p}+...+d_{m,p}
\end{equation*}
be the dimension of $\mathbb{X}_{p}$.
Consider a PLB probability measure $\mu_{p}$ on $\mathbb{X}_{p}$,
for every $p>0, \epsilon>0$, we define
\begin{equation}
E_{p}(\epsilon):=\bigcup_{\|\phi\|_{\mathscr{C}^{2}}\leq 1}\{x_{p}\in\mathbb{X}_{p}:
\bigl|\bigl<\Phi_{p}^{\star}(\delta_{x_{p}})-\Phi_{p}^{\star}(\mu_{p}), \phi\bigr>\bigr|\geq d(\Phi_{p})\epsilon\},
\end{equation}
where $\delta_{x_{p}}$ is the Dirac measure at the point $x_{p}$.
By the definition of the pullback of $\Phi_{p}$ on currents, we see that $\Phi_{p}^{\star}(\delta_{x_{p}})$
and $\Phi_{p}^{\star}(\mu_{p})$ are positive closed currents of bidimension $(k,k)$ on $X$.
Moreover, $\Phi_{p}^{\star}(\delta_{x_{p}})$ is well defined for $x_{p}\in\mathbb{X}_{p}$ generic.
The following estimate from Dinh-Sibony equidistribution theorem \cite{ds1} is crucial in our paper.
\begin{theorem}
Let $\eta_{\epsilon, p}:=\epsilon\delta(\Phi_{p})^{-1}d(\Phi_{p})-3R(\mathbb{X}_{p}, \omega_{p}, \mu_{p})$,
then
\begin{equation*}
\mu_{p}(E_{p}(\epsilon))\leq\Delta(\mathbb{X}_{p}, \omega_{p}, \mu_{p}, \eta_{\epsilon, p}).
\end{equation*}
\end{theorem}
We also need the following important estimate, which was deduced easily from \cite[Lemma 4.2(c),Proposition 4.3]{ds1}.
\begin{theorem}
In the above setting, we have
\begin{equation*}
\bigl|\bigl<d(\Phi_{p})^{-1}(\Phi_{p}^{\star}(\mu_{p})
-\Phi_{p}^{\star}(\omega_{p}^{d_{0,p}})), \phi\bigr>\bigr|\leq
2S(\mathbb{X}_{p}, \omega_{p}, \mu_{p})\delta(\Phi_{p})d(\Phi_{p})^{-1}\|\phi\|_{\mathscr{C}^{2}}
\end{equation*}
for any $(k,k)$-form $\phi$ of class $\mathscr{C}^{2}$ on $X$.
\end{theorem}

From now on, we study the special case when the meromorphic transforms are induced by Kodaira maps defined in Section 2.
We keep the settings from Section 1 in the sequel. Recall that the Kodaira map is defined in (9)
with an orthonormal basis $\{S_{k,p}^{j}\}_{j=0}^{d_{k,p}}$ on $\mathbb{P}H_{(2)}^{0}(X,L_{k}^{p})$.
We claim that the map is a meromorphic transform with the graph
\begin{equation}
\Gamma_{k,p}=\{(x,S)\in X\times\mathbb{P}H_{(2)}^{0}(X,L_{k}^{p}): S(x)=0\}.
\end{equation}
Since $\dim H_{(2)}^{0}(X,L_{k}^{p})\geq p^{n}/C\geq 2$ for some constant $C>0$ (see Theorem 4.1),
then for every point $x\in X$, there exists a section $S\in H_{(2)}^{0}(X,L_{k}^{p})$ such that $S(x)=0$.
Hence the projection from $\Gamma_{k,p}$ to $X$ is surjective.
Since $L_{k}^{p}$ is not trivial, then there are no nowhere vanishing sections.
That is to say, every global section $S\in H_{(2)}^{0}(X,L_{k}^{p})$ must vanish at some point $x\in X$.
Hence the projection from $\Gamma_{k,p}$ to $\mathbb{P}H_{(2)}^{0}(X,L_{k}^{p})$ is surjective.
The claim is proved. Note that
\begin{equation*}
\begin{split}
& \Phi_{k,p}(x)=\{S\in\mathbb{P}H_{(2)}^{0}(X,L_{k}^{p}): S(x)=0\}, \\
& \Phi_{k,p}^{-1}(S)=\{x\in X: S(x)=0\}. \\
\end{split}
\end{equation*}
Let
\begin{equation*}
\Phi_{p}: X\rightarrow\mathbb{X}_{p}
\end{equation*}
be the product map of $\Phi_{1,p},...,\Phi_{m,p}$.
We claim that $\Phi_{p}$ is also a meromorphic transform with the graph
\begin{equation*}
\Gamma_{p}=\{(x,S_{p1},...,S_{pm})\in X\times\mathbb{X}_{p}: S_{p1}(x)=...=S_{pm}(x)=0\}.
\end{equation*}
By the argument below (17), for every point $x\in X$ there exist $S_{pk}\in H_{(2)}^{0}(X,L_{k}^{p})$
such that $S_{pk}(x)=0$, $\forall 1\leq k\leq m$.
So the projection $\Pi_{1}: \Gamma_{p}\rightarrow X$ is surjective.
The projection $\Pi_{2}: \Gamma_{p}\rightarrow \mathbb{X}_{p}$ is proper,
then the image $\Pi_{2}(\Gamma_{p})$ is an analytic subvariety of $\mathbb{X}_{p}$ by Remmert's proper mapping theorem \cite{gr}.
Note that the zero set of every $S_{pk}\in H_{(2)}^{0}(X,L_{k}^{p})$ is represented by $Z(t_{k})$ for some $t_{k}$ defined in (13).
Then by Proposition 2.4 for almost every $(S_{p1},...,S_{pm})\in\mathbb{X}_{p}$ with respect to $\sigma_{p}$,
the common zero set of $S_{p1},...,S_{pm}$ is of pure dimension $n-m\geq 0$.
Then there exists some point $x\in X$ such that $(x,S_{p1},...,S_{pm})\in \Gamma_{p}$.
So $\sigma_{p}(\Pi_{2}(\Gamma_{p}))=1$. Hence $\Pi_{2}$ is surjective.
Indeed, if $\Pi_{2}$ is not surjective, then $\Pi_{2}(\Gamma_{p})$ is a proper analytic subvariety of $\mathbb{X}_{p}$,
Lemma 2.2 implies that $\sigma_{p}(\Pi_{2}(\Gamma_{p}))=0$, a contradiction.
Hence $\Phi_{p}$ is a meromorphic transform of codimension $n-m$ with fibers
\begin{equation*}
\Phi_{p}^{-1}(S_{p})=\{x\in X: S_{p1}(x)=...=S_{pm}(x)=0\},
\end{equation*}
where $S_{p}=(S_{p1},...,S_{pm})\in\mathbb{X}_{p}$.

Considering the product map of any $\Phi_{i_{1},p},...,\Phi_{i_{k},p}, 1\leq i_{1}<...<i_{k}\leq m$,
it follows from Proposition 2.4 that, the analytic subsets $(S_{p1}=0),...,(S_{pm}=0)$ are in general position
for $S_{p}=(S_{p1},...,S_{pm})\in\mathbb{X}_{p}$ generic. Then by \cite{de},
the current $[S_{p}=0]=[S_{p1}=0]\wedge...\wedge[S_{pm}=0]$ of bidegree $(m,m)$ is well defined for allmost all
$S_{p}\in\mathbb{X}_{p}$ with respect to $\sigma_{p}$.
\begin{lemma}
Denote by $\delta_{S_{p}}$ the Dirac measure at the point $S_{p}\in\mathbb{X}_{p}$,
then $\Phi_{p}^{\star}(\delta_{S_{p}})=[S_{p}=0]$.
\end{lemma}

\begin{proof}
By the definition of pullbacks of meromorphic transforms in (14),
we have
\begin{equation*}
\Phi_{p}^{\star}(\delta_{S_{p}})=\Pi_{1\star}(\Pi_{2}^{\star}(\delta_{S_{p}})\wedge[\Gamma_{p}]).
\end{equation*}
For any test $(n-m,n-m)$-form $\phi$ on $X$,
\begin{equation*}
\begin{split}
\big<\Phi_{p}^{\star}(\delta_{S_{p}}),\phi \bigr>&=\int_{X\times\mathbb{X}_{p}}\Pi_{2}^{\star}(\delta_{S_{p}})\wedge[\Gamma_{p}]\wedge\Pi_{1}^{\star}(\phi) \\
& =\int_{X\times\mathbb{X}_{p}}[\Pi_{2}^{-1}(S_{p})]\wedge[\Gamma_{p}]\wedge\Pi_{1}^{\star}(\phi) \\
& =\int_{\Pi_{2}^{-1}(S_{p})\cap\Gamma_{p}}\Pi_{1}^{\star}(\phi)=\int_{\{x\in X: S_{p}(x)=0\}}\phi \\
& =\bigl<[S_{p}=0], \phi \bigr>. \\
\end{split}
\end{equation*}
The proof is completed.
\end{proof}
\begin{remark}
Note that $\Phi_{k,p}^{\star}(\delta_{S_{pk}})=[S_{pk}=0]$ for each $1\leq k\leq m$.
The same argument yields
\begin{equation*}
\begin{split}
\Phi_{p}^{\star}(\delta_{S_{p}})&=[S_{p}=0]=[S_{p1}=0]\wedge...\wedge[S_{pm}=0] \\
& =\Phi_{1,p}^{\star}(\delta_{S_{p1}})\wedge...\wedge\Phi_{m,p}^{\star}(\delta_{S_{pm}}). \\
\end{split}
\end{equation*}
\end{remark}

Recall that $\pi_{k,p}: \mathbb{X}_{p}\rightarrow\mathbb{P}H_{(2)}^{0}(X, L_{k}^{p})$ is the natural projection.
Set
\begin{equation}
\omega_{p}:=c_{0,p}(\pi_{1,p}^{\star}\omega_{FS}+...+\pi_{m,p}^{\star}\omega_{FS}).
\end{equation}
We always assume that $\omega_{p}^{d_{0,p}}$ is a probability measure on $\mathbb{X}_{p}$.
Then $c_{0,p}$ satisfies the following condition,
\begin{equation}
(c_{0,p})^{-d_{0,p}}=\frac{d_{0,p}!}{d_{1,p}!...d_{m,p}!}.
\end{equation}
The sequence $\{c_{0,p}\}$ has a lower bound by using Stirling's formula (cf. \cite[p9]{dmm} and \cite[Lemma 4.3]{cmn}).
\begin{lemma}
There exists a positive constant $c_{0}$ such that $c_{0,p}\geq c_{0}$ for all $p\geq 1$.
\end{lemma}

To simplify the notations, we write
\begin{equation*}
\begin{split}
d_{p}&=d(\Phi_{p}):=\int_{X}\Phi_{p}^{\star}(\omega_{p}^{d_{0,p}})\wedge\omega^{n-m}, \\
\delta_{p}&=\delta(\Phi_{p}):=\int_{X}\Phi_{p}^{\star}(\omega_{p}^{d_{0,p}-1})\wedge\omega^{n-m+1}. \\
\end{split}
\end{equation*}
Using the classical cohomological arguments, $d_{p}$ and $\delta_{p}$ can be calculated as follows \cite[Lemma 4.4]{cmn}.
\begin{proposition}
In the above setting, we have
\begin{equation*}
\begin{split}
d_{p}&=p^{m}\|c_{1}(L_{1},h_{1})\wedge...\wedge c_{1}(L_{m},h_{m})\|, \\
\delta_{p}&=\frac{p^{m-1}}{c_{0,p}}\sum_{k=1}^{m}\frac{d_{k,p}}{d_{0,p}}\bigl\|\bigwedge_{l=1,l\neq k}^{m}c_{1}(L_{l},h_{l}) \bigr\|. \\
\end{split}
\end{equation*}
\end{proposition}

\begin{remark}
Lemma 3.6 implies that $\delta_{p}\leq Cp^{m-1}$ for some constant $C>0$ which depends on $(L_{k},h_{k}), 1\leq k\leq m$.
\end{remark}

Recall that $\gamma_{k,p}=\Phi_{k,p}^{\star}(\omega_{FS})$ is the Fubini-Study current in (11).
\begin{proposition} \emph{\cite[Lemma 4.5]{cmn}}
$\Phi_{p}^{\star}(\omega_{p}^{d_{0,p}})=\gamma_{1,p}\wedge...\wedge\gamma_{m,p}$ for all $p$ sufficiently large.
\end{proposition}

We recall the construction of moderate measures in the settings of Theorem 1.1.
Consider the functions $u_{j}:\mathbb{P}^{N}\rightarrow\mathbb{R}$, $1\leq j\leq N$. Fix an exponent $0<\rho<1$.
Let $\{u_{j}\}_{j=1}^{N}$ be a family of $(c_{N},\rho)$-functions defined above (7),
where $\{c_{N}\}_{N=1}^{\infty}$ is a sequence of positive numbers.
Set
\begin{equation}
\sigma_{N}:=\wedge_{j=1}^{N}(dd^{c}u_{j}+\omega_{FS}).
\end{equation}
This is a probability measure on $\mathbb{P}^{N}$.
We restate the following result from \cite[Remark 2.12]{sh}.
It shows that $\sigma_{N}$ is a moderate measure for suitable $c_{N}$ depending only on $\rho$ and $N$
(e.g. $c_{N}=O(1/c^{N})$, where the constant $c>1$ depends only on $\rho$).
\begin{proposition}
In the above setting, there exists a constant $0<c_{N}<1$ for the measure $\sigma_{N}$
which depends only on $\rho$ and $N$ such that
\begin{equation}
\int_{\mathbb{P}^{N}}\exp(-\alpha_{0}(\frac{\rho}{4})^{N}\phi)d\sigma_{N}\leq \beta_{0}N
\end{equation}
for all $\phi\in\mathcal{F}$,
where $\alpha_{0}<1, \beta_{0}$ are universal positive constants, and where $\mathcal{F}$ is defined in (4)
for $M=\mathbb{P}^{N}$ with the normalized Fubini-Study form.
That is to say, $\sigma_{N}$ is $(\beta_{0}N,\alpha_{0}(\frac{\rho}{4})^{N})$-moderate.
\end{proposition}
Now we are in a position to prove the important estimates on $R$ and $\Delta$.
\begin{proposition}
Under the above hypotheses, there exist universal positive constants $\beta_{1}, \beta_{2}, \beta_{3}$ such that $\forall t\in\mathbb{R}$,
\begin{equation*}
\begin{split}
R(\mathbb{P}^{N}, \omega_{FS}, \sigma_{N})&\leq \beta_{2}+\frac{1}{2}\log N, \\
\Delta(\mathbb{P}^{N}, \omega_{FS}, \sigma_{N}, t)&\leq \beta_{0}N\exp(-\alpha_{0}t)+
\beta_{1}(\frac{\rho}{4})^{N}\exp(-\alpha_{0}(\frac{\rho}{4})^{N}t). \\
\end{split}
\end{equation*}
When $t\leq \frac{1}{\alpha_{0}}(\log N+N\log\frac{4}{\rho})$, we have
\begin{equation*}
\Delta(\mathbb{P}^{N}, \omega_{FS}, \sigma_{N}, t)\leq \beta_{3}N\exp(-\alpha_{0}t).
\end{equation*}
\end{proposition}

\begin{proof}
By \cite[Proposition A.3]{ds1},
\begin{equation*}
R(\mathbb{P}^{N}, \omega_{FS}, \omega_{FS}^{N})\leq \frac{1}{2}(1+\log N).
\end{equation*}
We write
\begin{equation*}
\mu_{1,N}=\omega_{FS}^{N}, \quad \mu_{2,N}=\bigwedge_{j=1}^{N}(dd^{c}u_{j}+c_{N}\omega_{FS}+\omega_{FS})-\omega_{FS}^{N}.
\end{equation*}
The measure $\mu_{2,N}$ is positive since $u_{j}$ is $c_{N}\omega_{FS}$-p.s.h..
Recall that $\sigma_{N}$ is defined in (20), we have $\sigma_{N}\leq \mu_{1,N}+\mu_{2,N}$.
By \cite[Proposition 2.11]{sh},
there exists a universal positive constant $\beta_{1}$ such that for all $\phi\in\mathcal{F}$,
\begin{equation}
\int_{\mathbb{P}^{N}}\exp(-\alpha_{0}(\frac{\rho}{4})^{N}\phi)d\mu_{2,N}\leq \beta_{1}(\frac{\rho}{4})^{N}.
\end{equation}
Recall that $r(\mathbb{P}^{N}, \omega_{FS})=1$. It follows from the definitions of $R(\mathbb{P}^{N}, \omega_{FS}, \sigma_{N})$
in (15) and $\mathcal{F}$ in (4) that
\begin{equation*}
R(\mathbb{P}^{N}, \omega_{FS}, \sigma_{N})
=\sup_{\phi\in\mathcal{F}}\bigl\{-\int\phi d\sigma_{N}\bigr\}.
\end{equation*}
Since $\phi\leq 0, \sigma_{N}\leq \mu_{1,N}+\mu_{2,N}$ and $\sigma_{N}, \mu_{1,N}, \mu_{2,N}$ are all positive measures,
we have
\begin{equation*}
\begin{split}
R(\mathbb{P}^{N}, \omega_{FS}, \sigma_{N})
&\leq\sup_{\phi\in\mathcal{F}}\bigl\{-\int\phi d\mu_{1,N}-\int\phi d\mu_{2,N}\bigr\} \\
& \leq R(\mathbb{P}^{N}, \omega_{FS}, \omega_{FS}^{N})+\sup_{\phi\in\mathcal{F}}\bigl\{
\int\exp(-\alpha_{0}(\frac{\rho}{4})^{N}\phi)d\mu_{2,N}/(\alpha_{0}(\frac{\rho}{4})^{N}) \bigr\} \\
& \leq \frac{1}{2}(1+\log N)+\frac{\beta_{1}}{\alpha_{0}}\leq \beta_{2}+\frac{1}{2}\log N, \\
\end{split}
\end{equation*}
where the third inequality is deduced from (22) and \cite[Proposition 2.2]{sh}.
It follows from \cite[Inequality (13)]{sh} that
\begin{equation*}
\Delta(\mathbb{P}^{N}, \omega_{FS}, \sigma_{N}, t)
\leq \beta_{0}N\exp(-\alpha_{0}t)+
\beta_{1}(\frac{\rho}{4})^{N}\exp(-\alpha_{0}(\frac{\rho}{4})^{N}t).
\end{equation*}
Let
\begin{equation*}
N\exp(-\alpha_{0}t)=(\frac{\rho}{4})^{N}\exp(-\alpha_{0}(\frac{\rho}{4})^{N}t),
\end{equation*}
then
\begin{equation*}
t=\frac{\log N+N\log\frac{4}{\rho}}{\alpha_{0}(1-(\frac{\rho}{4})^{N})}.
\end{equation*}
Hence
\begin{equation*}
\Delta(\mathbb{P}^{N}, \omega_{FS}, \sigma_{N}, t)\leq \beta_{3}N\exp(-\alpha_{0}t),
\end{equation*}
when $t\leq \frac{1}{\alpha_{0}}(\log N+N\log\frac{4}{\rho})$.
This completes the proof.
\end{proof}

Now we study the estimates on multi-projective spaces.
Let $\mathbb{P}^{\ell_{1}},...,\mathbb{P}^{\ell_{m}}$ be $m$ projective spaces.
Let $\pi_{k}: \mathbb{P}^{\ell_{1}}\times...\times\mathbb{P}^{\ell_{m}}\rightarrow \mathbb{P}^{\ell_{k}}$
be the natural projection map.
Let $\sigma_{k}$ be a probability moderate measure with respect to a family of $(c_{\ell_{k}},\rho)$-functions
$\{u_{k,j}\}_{j=1}^{\ell_{k}}$ on $\mathbb{P}^{\ell_{k}}$ defined above (7). In the sequel of this section, $c_{\ell_{k}}$ is always chosen
such that the probability measure $\sigma_{k}$ satisfies the property of Proposition 3.10 (hence Proposition 3.11).
Let $\ell=\ell_{1}+...+\ell_{m}$ and
\begin{equation*}
\omega_{MP}:=c_{1m}(\pi_{1}^{\star}(\omega_{FS})+...+\pi_{m}^{\star}(\omega_{FS})), \quad c_{1m}^{-\ell}=\frac{\ell!}{\ell_{1}!...\ell_{m}!}.
\end{equation*}
It is equivalent to that $\omega_{MP}^{\ell}$ is a probability measure.
Recall that the notation $r(\mathbb{P}^{\ell_{1}}\times...\times \mathbb{P}^{\ell_{m}}, \omega_{MP})$
is defined after Lemma 3.1. We have the following lemma \cite[Lemma 4.6]{cmn}.
\begin{lemma}
Under the above hypotheses,
\begin{equation*}
r(\mathbb{P}^{\ell_{1}}\times...\times \mathbb{P}^{\ell_{m}}, \omega_{MP})\leq r(\ell_{1},...\ell_{m}):=\max_{1\leq k\leq m}\frac{\ell}{\ell_{k}}.
\end{equation*}
\end{lemma}

We first consider the case when $m=2$.
The corresponding result of estimates in a simpler case was proved in \cite[Proposition A.8]{ds1}.
Set $\omega_{12}:=\omega_{MP}$ as the the K\"{a}hler form on $\mathbb{P}^{\ell_{1}}\times\mathbb{P}^{\ell_{2}}$.
Denote by $\sigma$ the product of $\sigma_{1}$ and $\sigma_{2}$.
Write $r:=r(\mathbb{P}^{\ell_{1}}\times\mathbb{P}^{\ell_{2}}, \omega_{12})$.
Lemma 3.12 guarantees the existence of sufficiently large $\ell_{1}, \ell_{2}$ such that
\begin{equation*}
\frac{r\log(\ell_{1}+\ell_{2})}{\min(\ell_{1}, \ell_{2})}\ll 1.
\end{equation*}
\begin{proposition}
In the above setting, let $\mathbb{P}^{\ell_{1}}$ (resp. $\mathbb{P}^{\ell_{2}}$) be a projective space endowed
with a probability moderate measure $\sigma_{1}$ (resp. $\sigma_{2}$) satisfying Proposition 3.11.
Suppose that $\ell_{1}, \ell_{2}$ are chosen sufficiently large such that
\begin{equation}
\begin{split}
\frac{r\log(\ell_{1}+\ell_{2})}{\min(\ell_{1}, \ell_{2})}&\ll 1, \\
(\frac{\rho}{4})^{\min(\ell_{1}, \ell_{2})}(\ell_{1}+\ell_{2})&\ll 1. \\
\end{split}
\end{equation}
Then there exist universal positive constants $\beta_{4}, \beta_{5}$ such that for $0\leq t\leq \min(\ell_{1}, \ell_{2})$,
we have
\begin{equation}
\begin{split}
\Delta(\mathbb{P}^{\ell_{1}}\times\mathbb{P}^{\ell_{2}}, \omega_{12}, \sigma,t)&\leq
\beta_{4}(\ell_{1}+\ell_{2})^{1+\frac{\alpha_{0}}{2}}\exp(-\frac{\alpha_{0}}{2r}t), \\
R(\mathbb{P}^{\ell_{1}}\times\mathbb{P}^{\ell_{2}}, \omega_{12}, \sigma)
&\leq \beta_{5}r(1+\log (\ell_{1}+\ell_{2})). \\
\end{split}
\end{equation}
\end{proposition}

\begin{proof}
To simplify the notations, let $X_{1}=\mathbb{P}^{\ell_{1}}, X_{2}=\mathbb{P}^{\ell_{2}}, X=X_{1}\times X_{2}$.
Denote by $\omega_{1}$ (resp. $\omega_{2}$) the normalized Fubini-Study form $\omega_{FS}$ in $\mathbb{P}^{\ell_{1}}$
(resp. $\mathbb{P}^{\ell_{2}}$).
Consider a function $\psi$ on $X$ with the conditions that $\max_{X}\psi=0, dd^{c}\psi\geq -r\omega_{12}$.
Fix a point $(a,b)$ such that $\psi(a,b)=0$. Let $E$ be the set of all points with $\psi<-t$ for $t\geq 0$.
We write $E=(\psi<-t)$. Set
\begin{equation*}
\begin{split}
F:&=\{x_{2}\in X_{2}, \psi(a, x_{2})<-t/2\}, \\
E_{x_{2}}:&=\{x_{1}\in X_{1}, \psi(x_{1}, x_{2})<-t\}. \\
\end{split}
\end{equation*}
Define
\begin{equation*}
E':=\bigcup_{x_{2}\in X_{2}\setminus F}(E_{x_{2}}\times\{x_{2}\}).
\end{equation*}
Note that $E\subset\pi_{2}^{-1}(F)\cup E'$.

We first estimate the measure of $\pi_{2}^{-1}(F)$. Let $\psi_{1}(x_{2}):=\psi(a, x_{2})$,
then $\max_{X_{2}}\psi_{1}=\psi_{1}(b)=0$. Define a new function $\psi_{2}:=\psi_{1}-\int\psi_{1}d\sigma_{2}$.
Note that
\begin{equation*}
\int\psi_{2}d\sigma_{2}=0, \quad \psi_{2}\geq \psi_{1} \quad dd^{c}\psi_{2}\geq -r\omega_{2}.
\end{equation*}
Let $R:=\beta_{2}+\frac{1}{2}\log(\ell_{1}+\ell_{2})$.
Since $r(X_{2},\omega_{2})=1$, then by Proposition 3.11, we have
\begin{equation*}
-\int\psi_{1}d\sigma_{2}=\max_{X_{2}}\psi_{2}\leq rR(X_{2},\omega_{2},\sigma_{2})\leq rR.
\end{equation*}
Hence by hypotheses,
\begin{equation}
\begin{split}
\sigma_{2}(F)&\leq\sigma_{2}(\psi_{2}\leq rR-t/2) \\
& =\sigma_{2}(r^{-1}\psi_{2}\leq R-r^{-1}t/2)\leq\Delta(X_{2},\omega_{2},\sigma_{2},r^{-1}t/2-R) \\
& \leq\beta_{0}\ell_{2}\exp(\alpha_{0}R)\exp(-\frac{\alpha_{0}}{2r}t)+
\beta_{1}(\frac{\rho}{4})^{\ell_{2}}\exp(\alpha_{0}(\frac{\rho}{4})^{\ell_{2}}R)
\exp(-\frac{\alpha_{0}}{2r}(\frac{\rho}{4})^{\ell_{2}}t). \\
\end{split}
\end{equation}
When
\begin{equation*}
\frac{t}{2r}-R\leq\frac{1}{\alpha_{0}}(\log\ell_{2}+\ell_{2}\log\frac{4}{\rho}),
\end{equation*}
i.e.
\begin{equation*}
t\leq\frac{2r}{\alpha_{0}}\log\ell_{2}+\frac{2r}{\alpha_{0}}\ell_{2}\log\frac{4}{\rho}
+2r\beta_{2}+r\log(\ell_{1}+\ell_{2}),
\end{equation*}
it yields
\begin{equation}
\sigma_{2}(F)\leq\beta_{3}\ell_{2}\exp(\alpha_{0}R)\exp(-\frac{\alpha_{0}}{2r}t).
\end{equation}
Since $r\geq 1$ (cf. Lemma 3.12), $\alpha_{0}<1$, $\log\frac{4}{\rho}>1$,
(26) holds obviously when $0\leq t\leq\ell_{2}$.
By Fubini theorem, we obtain
\begin{equation}
\begin{split}
\sigma(\pi_{2}^{-1}(F))&\leq\beta_{0}\ell_{2}\exp(\alpha_{0}R)\exp(-\frac{\alpha_{0}}{2r}t) \\
&+\beta_{1}(\frac{\rho}{4})^{\ell_{2}}\exp(\alpha_{0}(\frac{\rho}{4})^{\ell_{2}}R)
\exp(-\frac{\alpha_{0}}{2r}(\frac{\rho}{4})^{\ell_{2}}t). \\
\end{split}
\end{equation}
When $0\leq t\leq\ell_{2}$,
\begin{equation}
\sigma(\pi_{2}^{-1}(F))\leq\beta_{3}\ell_{2}\exp(\alpha_{0}R)\exp(-\frac{\alpha_{0}}{2r}t).
\end{equation}

We secondly estimate the measure of $E'$. For any $x_{2}\in X_{2}\setminus F$, let $\psi_{3}(x_{1}):=\psi(x_{1},x_{2})$,
then $\psi_{3}\leq 0, \max_{X_{1}}\psi_{3}\geq\psi(a,x_{2})\geq -t/2$ and $dd^{c}\psi_{3}\geq -r\omega_{1}$.
Define a new function $\psi_{4}:=\psi_{3}-\int_{X_{1}}\psi_{3}d\sigma_{1}$.
Then
\begin{equation*}
\begin{split}
-\int\psi_{3}d\sigma_{1}&\leq\max_{X_{1}}\psi_{4}+t/2 \\
& \leq rR(X_{1},\omega_{1},\sigma_{1})+t/2\leq rR+t/2. \\
\end{split}
\end{equation*}
Hence by the same argument,
\begin{equation}
\begin{split}
\sigma_{1}(E_{x_{2}})&\leq\sigma_{1}(\psi_{4}\leq rR-t/2) \\
& \leq\beta_{0}\ell_{1}\exp(\alpha_{0}R)\exp(-\frac{\alpha_{0}}{2r}t) \\
&+\beta_{1}(\frac{\rho}{4})^{\ell_{1}}\exp(\alpha_{0}(\frac{\rho}{4})^{\ell_{1}}R)
\exp(-\frac{\alpha_{0}}{2r}(\frac{\rho}{4})^{\ell_{1}}t). \\
\end{split}
\end{equation}
When $0\leq t\leq\ell_{1}$,
\begin{equation}
\sigma_{1}(E_{x_{2}})\leq\beta_{3}\ell_{1}\exp(\alpha_{0}R)\exp(-\frac{\alpha_{0}}{2r}t).
\end{equation}
By Fubini theorem, we obtain
\begin{equation}
\begin{split}
\sigma(E')&\leq \beta_{0}\ell_{1}\exp(\alpha_{0}R)\exp(-\frac{\alpha_{0}}{2r}t) \\
& +\beta_{1}(\frac{\rho}{4})^{\ell_{1}}\exp(\alpha_{0}(\frac{\rho}{4})^{\ell_{1}}R)
\exp(-\frac{\alpha_{0}}{2r}(\frac{\rho}{4})^{\ell_{1}}t). \\
\end{split}
\end{equation}
When $0\leq t\leq\ell_{1}$,
\begin{equation}
\sigma(E')\leq\beta_{3}\ell_{1}\exp(\alpha_{0}R)\exp(-\frac{\alpha_{0}}{2r}t).
\end{equation}
So by estimates (27) and (31) for $t\geq 0$,
\begin{equation*}
\begin{split}
\sigma(\psi<-t)&\leq \beta_{0}(\ell_{1}+\ell_{2})\exp(\alpha_{0}R)\exp(-\frac{\alpha_{0}}{2r}t) \\
& +\sum_{j=1}^{2}\beta_{1}(\frac{\rho}{4})^{\ell_{j}}\exp(\alpha_{0}(\frac{\rho}{4})^{\ell_{j}}R)
\exp(-\frac{\alpha_{0}}{2r}(\frac{\rho}{4})^{\ell_{j}}t). \\
\end{split}
\end{equation*}
When $0\leq t\leq t_{0}=:\min(\ell_{1},\ell_{2})$, (28) and (32) yield
\begin{equation*}
\sigma(\psi<-t)\leq\beta_{3}(\ell_{1}+\ell_{2})\exp(\alpha_{0}R)\exp(-\frac{\alpha_{0}}{2r}t).
\end{equation*}
It is obvious that the above inequality is also valid for $t<0$ since $\psi\leq 0$.
By the definition of $\Delta(X, \omega_{12}, \sigma,t)$ in (15), we need to consider
a function $\varphi$ on $X$ with the conditions that $dd^{c}\varphi\geq -r\omega$ and $\int\varphi d\sigma=0$.
Define a new function $\psi:=\varphi-\max_{X}\varphi$. The fact that $\int\varphi d\sigma=0$ implies that $\max_{X}\varphi\geq 0$.
Then $\psi\leq\varphi$. Moreover, $\max_{X}\psi=0$. Then
\begin{equation*}
\sigma(\varphi<-t)\leq \sigma(\psi<-t).
\end{equation*}
Hence
\begin{equation}
\begin{split}
\Delta(X, \omega_{12}, \sigma,t)&\leq\beta_{0}(\ell_{1}+\ell_{2})\exp(\alpha_{0}R)\exp(-\frac{\alpha_{0}}{2r}t) \\
& +\sum_{j=1}^{2}\beta_{1}(\frac{\rho}{4})^{\ell_{j}}\exp(\alpha_{0}(\frac{\rho}{4})^{\ell_{j}}R)
\exp(-\frac{\alpha_{0}}{2r}(\frac{\rho}{4})^{\ell_{j}}t). \\
\end{split}
\end{equation}
When $0\leq t\leq t_{0}$,
\begin{equation}
\begin{split}
\Delta(X, \omega_{12}, \sigma,t)&\leq\beta_{3}(\ell_{1}+\ell_{2})\exp(\alpha_{0}R)\exp(-\frac{\alpha_{0}}{2r}t) \\
& =\beta_{4}(\ell_{1}+\ell_{2})^{1+\frac{\alpha_{0}}{2}}\exp(-\frac{\alpha_{0}}{2r}t), \\
\end{split}
\end{equation}
where $\beta_{4}=:\beta_{3}\exp(\alpha_{0}\beta_{2})$ is a universal positive constant.

To estimate $R(X, \omega_{12}, \sigma)$, we consider a function $\psi$ on $X$ with the conditions that $\max_{X}\psi=0, dd^{c}\psi\geq -r\omega_{12}$.
For any $0\leq \tilde t\leq t_{0}$,
\begin{equation}
\begin{split}
-\int\psi d\sigma &=\int_{0}^{\infty}\sigma(\psi\leq -t)dt \\
& =\int_{0}^{\tilde t}\sigma(\psi\leq -t)dt+\int_{\tilde t}^{t_{0}}\sigma(\psi\leq -t)dt+\int_{t_{0}}^{\infty}\sigma(\psi\leq -t)dt \\
& \leq \int_{0}^{\tilde t}dt+\int_{\tilde t}^{\infty}(\beta_{0}+\beta_{3})(\ell_{1}+\ell_{2})\exp(\alpha_{0}R)\exp(-\frac{\alpha_{0}}{2r}t)dt \\
& +\sum_{j=1}^{2}\beta_{1}\int_{t_{0}}^{\infty}(\frac{\rho}{4})^{\ell_{j}}\exp(\alpha_{0}(\frac{\rho}{4})^{\ell_{j}}R)
\exp(-\frac{\alpha_{0}}{2r}(\frac{\rho}{4})^{\ell_{j}}t)dt \\
& =\tilde t+\frac{2r}{\alpha_{0}}(\beta_{0}+\beta_{3})(\ell_{1}+\ell_{2})\exp(\alpha_{0}R)\exp(-\frac{\alpha_{0}}{2r}\tilde t) \\
& +\sum_{j=1}^{2}\frac{2r}{\alpha_{0}}\beta_{1}\exp(\alpha_{0}
(\frac{\rho}{4})^{\ell_{j}}R)\exp(-\frac{\alpha_{0}}{2r}(\frac{\rho}{4})^{\ell_{j}}t_{0}). \\
\end{split}
\end{equation}
The above inequality follows from (33) and (34).
By the hypotheses in (23), the last term in the last equality is less than $\frac{5r}{\alpha_{0}}\beta_{1}$
for $\ell_{1},\ell_{2}$ sufficiently large.
Hence
\begin{equation*}
-\int\psi d\sigma\leq\tilde t+\frac{2r}{\alpha_{0}}(\beta_{0}+\beta_{3})
(\ell_{1}+\ell_{2})\exp(\alpha_{0}R)\exp(-\frac{\alpha_{0}}{2r}\tilde t)+\frac{5r}{\alpha_{0}}\beta_{1}.
\end{equation*}
Take $\tilde t=2rR+\frac{2r}{\alpha_{0}}\log ((\beta_{0}+\beta_{3})(\ell_{1}+\ell_{2}))$.
By the hypotheses in (23), $\tilde t\leq t_{0}$ for $\ell_{1},\ell_{2}$ sufficiently large.
We deduce that
\begin{equation*}
\begin{split}
 -\int\psi d\sigma&\leq 2rR+2rR+\frac{2r}{\alpha_{0}}\log ((\beta_{0}+\beta_{3})(\ell_{1}+\ell_{2}))+
\frac{2r}{\alpha_{0}}+\frac{5r}{\alpha_{0}}\beta_{1} \\
& \leq\beta_{5}r(1+\log (\ell_{1}+\ell_{2})), \\
\end{split}
\end{equation*}
where $\beta_{5}$ is a universal positive constant.
This completes the proof.
\end{proof}

The following proposition shows the main estimates in this section.
\begin{proposition}
In the above setting, let $\mathbb{P}^{\ell_{k}}$ be a projective space endowed
with a probability moderate measure $\sigma_{k}$ satisfying Proposition 3.11, $\forall 1\leq k\leq m$.
Set $\sigma:=\sigma_{1}\times...\times\sigma_{m}$.
Suppose that $\ell_{1},..,\ell_{m}$ are chosen sufficiently large such that
\begin{equation}
\begin{split}
\frac{r(\ell_{1},...,\ell_{m})\log\ell}{\min(\ell_{1},...,\ell_{m})}&\ll 1, \\
(\frac{\rho}{4})^{\min(\ell_{1},...,\ell_{m})}\ell&\ll 1. \\
\end{split}
\end{equation}
Then there exist positive constants $\beta_{6}, \beta, \xi$ depending only on $m$ such that for $0\leq t\leq \min(\ell_{1},...,\ell_{m})$,
we have
\begin{equation*}
\begin{split}
R(\mathbb{P}^{\ell_{1}}\times...\times \mathbb{P}^{\ell_{m}}, \omega_{MP}, \sigma)&\leq \beta_{6}r(\ell_{1},...\ell_{m})(1+\log\ell), \\
S(\mathbb{P}^{\ell_{1}}\times...\times \mathbb{P}^{\ell_{m}}, \omega_{MP}, \sigma)&\leq \beta_{6}r(\ell_{1},...\ell_{m})(1+\log\ell), \\
\Delta(\mathbb{P}^{\ell_{1}}\times...\times \mathbb{P}^{\ell_{m}}, \omega_{MP}, \sigma, t)&\leq \beta_{6}\ell^{\xi}\exp(-\beta t/r(\ell_{1},...\ell_{m})). \\
\end{split}
\end{equation*}
\end{proposition}

\begin{proof}
When $m=2$, the estimates on $R$ and $\Delta$ are proved in Proposition 3.13.
When $m=3$, following the notations in the proof of Proposition 3.13, we write
$X_{1}=\mathbb{P}^{\ell_{1}}\times \mathbb{P}^{\ell_{2}}, X_{2}=\mathbb{P}^{\ell_{3}}, X=X_{1}\times X_{2}$.
The estimates on $R$ and $\Delta$ for $X_{1}$ (resp. $X_{2}$) are showed in Proposition 3.13 and (33)
(resp. Proposition 3.11). Consequently, the results of estimates on $R$ and $\Delta$ for $X$
are proved by using the analogous arguments in (25),(26),(29),(30) and (35) with the hypotheses (36).
For the general case, the results can be deduced inductively by using the analogous arguments
in the proof of Proposition 3.13.
The estimate on $S$ follows from \cite[Proposition 2.4]{ds1} and \cite[Lemma 4.6]{cmn}.
\end{proof}

\section{Proof of the main theorems}
In this section we will prove the main theorems.
First we give an estimate of the dimension $d_{k,p}$.
The lower estimate is proved by construction of a new metric on the line bundle
with only one singularity and application of vanishing theorem relative to multiplier ideal sheaves,
see \cite[Proposition 4.7]{cmn}, \cite[Theorem 4.5]{de2} and \cite[2.3.28]{mm1}.

\begin{theorem}
Let $(X,\omega)$ be a compact K\"{a}hler manifold of dimension $n$.
Suppose that $(L,h)$ is a singular Hermitian holomorphic line bundle on $X$ such that
$c_{1}(L,h)\geq\epsilon\omega$ for some positive constant $\epsilon$.
Moreover, $h$ is continuous outside a proper analytic subset $A$ of $X$.
Then there exist a constant $C>1$ and $p_{0}\in\mathbb{N}$ such that
for all $p\geq p_{0}$
\begin{equation*}
p^{n}/C\leq\dim H_{(2)}^{0}(X,L^{p})\leq Cp^{n}.
\end{equation*}
\end{theorem}

Recall that $\gamma_{k,p}$ is the Fubini-Study current defined in (11).
With techniques from \cite{cm1} and \cite{fs1}, we can show that the sequence of wedge products of
these Fubini-Study currents converges weakly to the wedge product of the curvature currents of
the line bundles in Theorem 1.1, see \cite[Proposition 3.1]{cmn}.

\begin{proposition}
In the setting of Theorem 1.1, we have
\begin{equation*}
\frac{1}{p^{m}}\gamma_{1,p}\wedge...\wedge\gamma_{m,p}\rightarrow c_{1}(L_{1},h_{1})\wedge...\wedge c_{1}(L_{m},h_{m})
\end{equation*}
when $p$ tends to $\infty$.
\end{proposition}

We also need the following convergence property.
\begin{proposition}
In the setting of Theorem 1.1, there exists a positive constant $C$ depending only on $X, (L_{1},h_{1}),...,(L_{m},h_{m})$ such that
\begin{equation*}
\bigl|\bigl<\frac{1}{p^{m}}(\Phi_{p}^{\star}(\sigma_{p})-\Phi_{p}^{\star}(\omega_{p}^{d_{0,p}})),\phi \bigr>\bigr|
\leq \frac{C\log p}{p}\|\phi\|_{\mathscr{C}^{2}}
\end{equation*}
for any $(n-m,n-m)$-form of class $\mathscr{C}^{2}$ on $X$ and $p$ sufficiently large.
In particular, $\frac{1}{p^{m}}(\Phi_{p}^{\star}(\sigma_{p})-\Phi_{p}^{\star}(\omega_{p}^{d_{0,p}}))$
converges weakly to $0$ as $p\rightarrow\infty$.
\end{proposition}

\begin{proof}
By Theorem 4.1, there exist a positive constant $C_{1}$
and $p_{0}\in\mathbb{N}$ such that for all $p\geq p_{0}, 1\leq k\leq m$, we have
\begin{equation*}
p^{n}/C_{1}\leq d_{k,p}\leq C_{1}p^{n}.
\end{equation*}
Then by Lemma 3.12, $r(\mathbb{X}_{p}, \omega_{p})\leq mC_{1}^{2}$.
Moreover, $d_{1,p},...,d_{m,p}$ satisfy the conditions in (36) for $p$ sufficiently large.
Hence it follows from Proposition 3.14 that there exists a positive constant $C_{2}$,
\begin{equation*}
S(\mathbb{X}_{p}, \omega_{p}, \sigma_{p})\leq C_{2}\log p.
\end{equation*}
Thanks to Proposition 3.7 and Remark 3.8, we can deduce that
\begin{equation*}
\delta_{p}d_{p}^{-1}\leq C_{3}\frac{1}{p}
\end{equation*}
for some positive constant $C_{3}$.
Note that the constants $C_{1},C_{2},C_{3}$ all depend only on $X, (L_{1},h_{1}),...,(L_{m},h_{m})$.
Then
\begin{equation*}
2S(\mathbb{X}_{p}, \omega_{p}, \sigma_{p})\delta_{p}d_{p}^{-1}\leq \frac{C\log p}{p}
\end{equation*}
for some positive constant $C$ depending only on $X, (L_{1},h_{1}),...,(L_{m},h_{m})$
Hence the proof is completed by applying Theorem 3.3.
\end{proof}

The basic proof of the first main theorem will end with the following theorem,
which extends \cite[Corollary 3.9]{dmm} and \cite[Theorem 4.2]{cmn}
\begin{theorem}
In the setting of Theorem 1.1, there exist a positive constant $\xi$ which depends only on $m$
and a positive constant $C$ which depends only on $X, (L_{1},h_{1}),...,(L_{m},h_{m})$ with the following property:
Given any sequence of positive numbers $\{\lambda_{p}\}_{p=1}^{\infty}$ with the following conditions
\begin{equation*}
\begin{split}
\liminf_{p\rightarrow\infty}\frac{\lambda_{p}}{\log p}&>(1+\xi n)C, \\
\lim_{p\rightarrow\infty}\frac{\lambda_{p}}{p^{n}}&=0, \\
\end{split}
\end{equation*}
there exist subsets $E_{p}\subset\mathbb{X}_{p}$ such that for all $p$ sufficiently large,
\par (i)
\begin{equation*}
\sigma_{p}(E_{p})\leq Cp^{\xi n}\exp(-\frac{\lambda_{p}}{C}),
\end{equation*}
\par (ii) for any point $S_{p}\in\mathbb{X}_{p}\setminus E_{p}$ and any $(n-m,n-m)$-form $\phi$ of class $\mathscr{C}^{2}$,
\begin{equation*}
\bigl|\frac{1}{p^{m}}\bigl<[S_{p}=0]-\Phi_{p}^{\star}(\sigma_{p}), \phi \bigr> \bigr|\leq \frac{C\lambda_{p}}{p}\|\phi\|_{\mathscr{C}^{2}}.
\end{equation*}
\end{theorem}

\begin{proof}
To simplify the notations, let
\begin{equation}
\begin{split}
R_{p}:&=R(\mathbb{X}_{p},\omega_{p},\sigma_{p}), \\
\Delta_{p}(t):&=\Delta(\mathbb{X}_{p},\omega_{p},\sigma_{p},t), \\
E_{p}(\epsilon):&=\bigcup_{\|\phi\|_{\mathscr{C}^{2}}\leq 1}\{S_{p}\in\mathbb{X}_{p}:
|\bigl<[S_{p}=0]-\Phi_{p}^{\star}(\sigma_{p}), \phi \bigr>|\geq d_{p}\epsilon\},
\end{split}
\end{equation}
where $t\geq 0, \epsilon>0$.
By \cite[Theorem 1.1]{sh} with its proof, there exists a constant $c>1$ which depends only on
$X, L_{1},...,L_{m}, \rho$ such that each component $\bigwedge_{j=1}^{d_{k,p}}\pi_{k,p}^{\star}(dd^{c}u_{j}^{k,p}+\omega_{FS})$
of $\sigma_{p}$ is a probability moderate measure satisfying Proposition 3.11.
Theorem 4.1 implies that $d_{1,p},...,d_{m,p}$ satisfy the conditions in (36) for p sufficiently large.
Hence $\sigma_{p}$ satisfy Proposition 3.14.
Let $\tilde C$ be a positive constant depending only on $X, (L_{1},h_{1}),...,(L_{m},h_{m})$ such that
for all $p\geq p_{0}, 1\leq k\leq m$, we have
\begin{equation*}
p^{n}/\tilde C\leq d_{k,p}\leq \tilde Cp^{n},
\end{equation*}
Here $p_{0}$ is a positive integer which is large enough.
Then we have for $p\geq p_{0}$ and $0\leq t\leq p^{n}/\tilde C$,
\begin{equation}
\begin{split}
R_{p}&\leq m\beta_{6}\tilde C^{2}(1+\log(m\tilde Cp^{n}))\leq C_{1}\log p, \\
\Delta_{p}(t)&\leq \beta_{6}(m\tilde Cp^{n})^{\xi}\exp(\frac{-\beta t}{m\tilde C^{2}})\leq C_{1}p^{\xi n}\exp(-\frac{t}{C_{1}}).
\end{split}
\end{equation}
Here $C_{1}$ is some constant depending only on $X, (L_{1},h_{1}),...,(L_{m},h_{m})$.
Let
\begin{equation}
\epsilon_{p}:=\frac{\lambda_{p}}{p}, \quad \eta_{p}:=\epsilon_{p}d_{p}\delta_{p}^{-1}-3R_{p}.
\end{equation}
It follows from Proposition 3.7 and Remark 3.8 that for $p\geq p_{0}$,
\begin{equation*}
\eta_{p}\geq C_{2}\lambda_{p}-3C_{1}\log p.
\end{equation*}
Here $C_{2}$ is some constant depending only on $X, (L_{1},h_{1}),...,(L_{m},h_{m})$.
If there is a condition that
\begin{equation*}
\liminf_{p\rightarrow\infty}\frac{\lambda_{p}}{\log p}>\frac{6C_{1}}{C_{2}},
\end{equation*}
then for all $p$ sufficiently large, $\eta_{p}>\frac{C_{2}}{2}\lambda_{p}$.
Since
\begin{equation*}
\lim_{p\rightarrow\infty}\frac{\lambda_{p}}{p^{n}}=0,
\end{equation*}
$\eta_{p}$ can be always chosen such that $\frac{C_{2}}{2}\lambda_{p}<\eta_{p}<p^{n}/\tilde C$ for $p$ sufficiently large.
By applying Theorem 3.2 to the subset $E_{p}\subset\mathbb{X}_{p}$, we obtain
\begin{equation*}
\sigma_{p}(E_{p})\leq \Delta_{p}(\eta_{p})\leq C_{1}p^{\xi n}\exp(\frac{-C_{2}}{2C_{1}}\lambda_{p}),
\end{equation*}
where $E_{p}=E_{p}(\epsilon_{p})$.
Now we set
\begin{equation*}
C=\max\bigl(\frac{6C_{1}}{C_{2}(1+\xi n)},\frac{2C_{1}}{C_{2}},C_{1}, \|c_{1}(L_{1},h_{1})\wedge...\wedge c_{1}(L_{m},h_{m})\| \bigr).
\end{equation*}
Then for all $p$ sufficiently large,
\begin{equation*}
\sigma_{p}(E_{p})\leq Cp^{\xi n}\exp(\frac{-\lambda_{p}}{C})
\end{equation*}
under the conditions that
\begin{equation*}
\begin{split}
\liminf_{p\rightarrow\infty}\frac{\lambda_{p}}{\log p}&>(1+\xi n)C, \\
\lim_{p\rightarrow\infty}\frac{\lambda_{p}}{p^{n}}&=0. \\
\end{split}
\end{equation*}
By the definition of $E_{p}$, it is obvious that for any $S_{p}\in\mathbb{X}_{p}\setminus E_{p}$
and any $(n-m,n-m)$-form $\phi$ of class $\mathscr{C}^{2}$,
\begin{equation}
\bigl|\frac{1}{p^{m}}\bigl<[S_{p}=0]-\Phi_{p}^{\star}(\sigma_{p}), \phi \bigr> \bigr|
\leq\frac{d_{p}}{p^{m}}\frac{\lambda_{p}}{p}\|\phi\|_{\mathscr{C}^{2}}
\leq \frac{C\lambda_{p}}{p}\|\phi\|_{\mathscr{C}^{2}}.
\end{equation}
This completes the proof.
\end{proof}

\begin{proposition}
In the setting of Theorem 1.1, for almost all $S=\{S_{p}\}_{p=1}^{\infty}\in\mathbb{P}^{X}$ with respect to $\sigma$,
we have
\begin{equation*}
\frac{1}{p^{m}}[S_{p}=0]-\frac{1}{p^{m}}\Phi_{p}^{\star}(\sigma_{p})\rightarrow 0
\end{equation*}
in the weak sense of currents as $p\rightarrow \infty$ on $X$.
\end{proposition}

\begin{proof}
This is a standard proof which is analogous to that of Borel-Cantelli lemma.
It follows from Theorem 4.4 that
\begin{equation*}
\sum_{p=1}^{\infty}\sigma_{p}(E_{p})\leq C_{3}\sum_{p=1}^{\infty}\frac{1}{p^{\eta}}<\infty
\end{equation*}
for some positive constant $C_{3}$ and $\eta>1$, where $E_{p}=E_{p}(\epsilon_{p})$ (cf. (37) and (39)).
Define
\begin{equation*}
E:=\{S=\{S_{p}\}_{p=1}^{\infty}\in\mathbb{P}^{X}: S_{p}\in E_{p} \quad \text{for infinitely many}\quad p\}.
\end{equation*}
It is easy to see that $E$ is contained in the following set
\begin{equation*}
\tilde E_{N}:=\{S=\{S_{p}\}_{p=1}^{\infty}\in\mathbb{P}^{X}: S_{p}\in E_{p} \quad\text{for at least one}\quad p\geq N\}
\end{equation*}
for each integer $N\geq 1$.
Hence we have
\begin{equation*}
\sigma(E)\leq \sigma(\tilde E_{N})\leq \sum_{p=N}^{\infty}\sigma_{p}(E_{p})\leq C_{3}\sum_{p=N}^{\infty}\frac{1}{p^{\eta}}.
\end{equation*}
The proof is completed by letting $N$ tend to $\infty$.
\end{proof}

\noindent $\textbf{End of the proof of Theorem 1.1}$.
By \cite[Theorem 1.1]{sh} with its proof, there exists a constant $c>1$ which depends only on
$X, L_{1},...,L_{m}, \rho$ such that each component $\bigwedge_{j=1}^{d_{k,p}}\pi_{k,p}^{\star}(dd^{c}u_{j}^{k,p}+\omega_{FS})$
of $\sigma_{p}$ is a probability moderate measure satisfying Proposition 3.11.
Hence $\sigma_{p}$ satisfy Proposition 3.14. Note that $c$ is independent of the choices of the metrics $h_{1},...,h_{m}$.
It follows from Proposition 3.9 and Proposition 4.3 that
\begin{equation*}
\frac{1}{p^{m}}(\Phi_{p}^{\star}(\sigma_{p})-\gamma_{1,p}\wedge...\wedge\gamma_{m,p})\rightarrow 0
\end{equation*}
in the weak sense of currents as $p\rightarrow\infty$.
Then Proposition 4.5 implies that for almost all $S=\{S_{p}\}_{p=1}^{\infty}\in\mathbb{P}^{X}$ with respect to $\sigma$
\begin{equation*}
\frac{1}{p^{m}}([S_{p}=0]-\gamma_{1,p}\wedge...\wedge\gamma_{m,p})\rightarrow 0
\end{equation*}
in the weak sense of currents as $p\rightarrow\infty$.
The proof is finally completed by  application of Proposition 4.2. $\vspace{12pt}$\qquad \qquad \qquad \qquad \qquad \qquad
\qquad \qquad \qquad \qquad \qquad \quad $\qed$

Theorem 1.3 follows from Proposition 4.3 and Theorem 4.4 with the following theorem \cite[Proposition 5.1]{cmn}.

\begin{theorem}
In the setting of Theorem 1.3, there exists a positive constant
$C$ which depends only on $X, (L_{1},h_{1}),...,(L_{m},h_{m})$ such that
for all $p$ sufficiently large and any $(n-m,n-m)$-form $\phi$ of class $\mathscr{C}^{2}$,
we have
\begin{equation*}
\bigl|\bigl<\frac{1}{p^{m}}\gamma_{1,p}\wedge...\wedge\gamma_{m,p}-c_{1}(L_{1},h_{1})\wedge...\wedge c_{1}(L_{m},h_{m})
, \phi \bigr> \bigr|\leq \frac{C\log p}{p}\|\phi\|_{\mathscr{C}^{2}}.
\end{equation*}
\end{theorem}

To prove Theorem 1.5, we need the following result \cite[Theorem 15.1.6]{hl}.
\begin{theorem}
Let $P_{A}$ be the set of all functions of the form $p^{-1}\log |f(z)|$
where $p$ is a positive integer and $f$ an entire function $\not \equiv 0$ in $\mathbb{C}^{n}$.
Then the closure of $P_{A}$ in $L^{1}_{loc}(\mathbb{C}^{n})$ consists of
all plurisubharmonic functions.
\end{theorem}
Let $\omega_{0}$ be the Fubini-Study form with mass $1$ in $\mathbb{P}^{n}$.
Given any positive closed current $T$ of bidegree $(1,1)$ with mass $1$ in $\mathbb{P}^{n}$.
Then by $dd^{c}$-lemma \cite[Lemma 8.6]{de3},
there exists a q.p.s.h. function $\varphi$ such that
\begin{equation*}
T-\omega_{0}=dd^{c}\varphi.
\end{equation*}
Therefore $T$ corresponds to an entire plurisubharmonic function $\psi$
in $\mathbb{C}^{n}$ which belongs to the Lelong class (cf. \cite[Example 2.2]{gz}). By Theorem 4.7,
there exists a sequence $\{p^{-1}\log |f_{p}|\}$ which converges to $\psi$.
Since holomorphic functions in $\mathbb{C}^{n}$ can be approximated by polynomials,
by using diagonal argument, we can choose a sequence of polynomials $g_{p}$ of degree $\leq p$
such that $\{p^{-1}\log |g_{p}|\}$ converges to $\psi$. It is possible since all such $p^{-1}\log |g_{p}|$ and
$\psi$ belong to the Lelong class.
Note that $g_{p}$ can be regarded as a homogeneous polynomial of degree $p$ in $\mathbb{C}^{n+1}$,
it induces a global section $S_{p}\in H^{0}(\mathbb{P}^{n}, \mathcal{O}(p))$.
Hence by Lelong-Poincar\'{e} formula, $\frac{1}{p}[S_{p}=0]$ converges weakly to $T$.
We obtain the following result due to Oka.
\begin{proposition}
Given any positive closed current $T$ of bidegree $(1,1)$ with mass $1$ in $\mathbb{P}^{n}$,
there exists a sequence of $\{S_{p}\}_{p\geq 1}$, $S_{p}\in\mathbb{P}H^{0}(\mathbb{P}^{n}, \mathcal{O}(p))$,
such that
\begin{equation*}
\frac{1}{p}[S_{p}=0]\rightarrow T
\end{equation*}
in the weak sense of currents.
\end{proposition}
This section concludes with the proof of Theorem 1.5.
\begin{proof}
By Proposition 4.8, there exists a sequence of $\{S_{p}\}_{p\geq 1}$, $S_{p}\in\mathbb{P}H^{0}(\mathbb{P}^{n}, \mathcal{O}(p))$,
such that
\begin{equation*}
\frac{1}{p}[S_{p}=0]\rightarrow T.
\end{equation*}
Denote by $\delta_{S_{p}}$ the Dirac measure at the point $S_{p}\in\mathbb{P}H^{0}(\mathbb{P}^{n}, \mathcal{O}(p))$.
Choose a sequence of smooth probability measures $\{\mu_{pj}\}$ which is an approximation of $\delta_{S_{p}}$.
Note that the map
\begin{equation*}
\begin{split}
\mathbb{P}H^{0}(\mathbb{P}^{n}, \mathcal{O}(p))&\rightarrow\mathbb{C} \\
V_{p}&\mapsto\bigl<[V_{p}=0],\phi \bigr>
\end{split}
\end{equation*}
is continuous. Then there exists a neighborhood $E_{p}\subset\mathbb{P}H^{0}(\mathbb{P}^{n}, \mathcal{O}(p))$ of $S_{p}$
such that
\begin{equation*}
|\bigl<[V_{p}=0],\phi \bigr>-\bigl<[S_{p}=0],\phi \bigr>|\leq 1,
\end{equation*}
$\forall V_{p}\in E_{p}$ and $(n-1,n-1)$-form $\phi$ of class $\mathscr{C}^{2}$ with $\|\phi\|_{\mathscr{C}^{2}}\leq 1$.
Hence
\begin{equation*}
\begin{split}
& |\bigl<\frac{1}{p}[V_{p}=0]-T,\phi \bigr>| \\
& \leq |\bigl<\frac{1}{p}[V_{p}=0]-\frac{1}{p}[S_{p}=0],\phi \bigr>|+|\bigl<\frac{1}{p}[S_{p}=0]-T,\phi \bigr>| \\
& \leq\frac{1}{p}\|\phi\|_{\mathscr{C}^{2}}+|\bigl<\frac{1}{p}[S_{p}=0]-T,\phi \bigr>|\rightarrow 0, \\
\end{split}
\end{equation*}
$\forall V_{p}\in E_{p}$ and $(n-1,n-1)$-form $\phi$ of class $\mathscr{C}^{2}$.
Since $\mu_{pj}\rightarrow\delta_{S_{p}}$ as measures when $j\rightarrow\infty$, there exists
an index $j_{p}$ satisfying
\begin{equation*}
|\mu_{pj_{p}}(E_{p})-\delta_{S_{p}}(E_{p})|=|\mu_{pj_{p}}(E_{p})-1|\leq\frac{1}{p^{2}}.
\end{equation*}
Denote by $E_{p}^{c}$ the complement of $E_{p}$.
Set $\sigma_{p}:=\mu_{pj_{p}}$. So $\sigma_{p}(E_{p}^{c})\leq\frac{1}{p^{2}}$.
Yau's theorem \cite{y} implies that there exists a smooth real function $u_{p}$ with
$(\omega_{FS}+dd^{c}u_{p})^{d_{p}}=\sigma_{p}$.
Then the theorem follows from the same argument in the proof of Proposition 4.5.
This completes the proof.
\end{proof}

\noindent
G. Shao,
Universit{\'e} Paris-Sud, Math{\'e}matique - B{\^a}timent 425, 91405
Orsay, France. {\tt shaoguokuang@gmail.com}

\end{document}